\documentclass[onefignum,onetabnum]{siamart190516}

\usepackage{lipsum}
\usepackage{amsfonts}
\usepackage{graphicx}
\usepackage{epstopdf}
\usepackage{algorithmic}
\usepackage{amsmath}
\usepackage{subfigure}	
\usepackage{multirow}

\ifpdf
  \DeclareGraphicsExtensions{.eps,.pdf,.png,.jpg}
\else
  \DeclareGraphicsExtensions{.eps}
\fi

\newsiamremark{remark}{Remark}
\newsiamremark{example}{Example}
\newsiamremark{hypothesis}{Hypothesis}
\crefname{hypothesis}{Hypothesis}{Hypotheses}
\newsiamthm{claim}{Claim}

\headers{Parametric resonance for enhancing the rate of transition}{Y. Chao and M. Tao}

\title{Parametric resonance for enhancing the rate of metastable transition\thanks{Submitted to the editors DATE.
\funding{This work was partly supported by NSFC grant 12101484, NSF DMS-1847802.}}}

\author{Ying Chao\thanks{School of Mathematics and Statistics, Xi'an Jiaotong University, Xi'an, Shaanxi 710049,  P.R. China 
  (\email{yingchao1993@xjtu.edu.cn}).}
\and Molei Tao\thanks{School of Mathematics, Georgia Institute of Technology, Atlanta, GA 30332, USA 
  (\email{mtao@gatech.edu}) Corresponding author.}}
\usepackage{amsopn}

\ifpdf
\hypersetup{
  pdftitle={Parametric resonance for enhancing the rate of metastable transition},
  pdfauthor={Ying Chao and Molei Tao}
}
\fi

\theoremstyle{definition}

\begin{document}

\maketitle

\begin{abstract}
This work is devoted to quantifying how periodic perturbation can change the rate of metastable transition in stochastic mechanical systems with weak noises. A closed-form explicit expression for approximating the rate change is provided, and the corresponding transition mechanism can also be approximated. Unlike the majority of existing relevant works, these results apply to kinetic Langevin equations with high-dimensional potentials and nonlinear perturbations. They are obtained based on a higher-order Hamiltonian formalism and  perturbation analysis for the Freidlin-Wentzell action functional. This tool allowed us to show that parametric excitation at a resonant frequency can significantly enhance the rate of metastable transitions. Numerical experiments for both  low-dimensional toy models and a molecular cluster are also provided. For the latter, we show that vibrating a material appropriately can help heal its defect, and our theory provides the appropriate vibration.

\end{abstract}



\section{Introduction}
\label{sec:intro}
Rare but reactive dynamical events induced by small noise underlie many physical, chemical and biological problems. Examples of such rare events include climate changes (e.g., \cite{ragone2018computation}), nucleation in phase transitions (e.g., \cite{Heymann2008Pathways}), activated chemical reactions and conformation switching of macromolecules (e.g., \cite{weinan2002string}). To explore the mechanism of rare transition in stochastic dynamical systems is a challenging task. In the limit of weak noise, Freidlin–Wentzell large deviation theory \cite{Freidlin2012, Dembo2010} provides a framework for assessing the likelihoods of those rare events. 

This paper considers a specific case of non-autonomously forced stochastic mechanical system, modeled as a second-order\footnote{We call it 2nd-order because it can be formally written as $\ddot{x}+\Gamma\dot{x}=-\nabla V(x)+\varepsilon f(x,t)+\sqrt{\mu} \Gamma^{\frac{1}{2}}\xi(t)$, which should be compared with the 1st-order system of perturbed overdamped Langevin, i.e., $\dot{x}=-\nabla V(x)+\varepsilon f(x,t)+\sqrt{\mu} \xi(t)$. Here, $\xi(t)$ is a standard white noise in $\mathbb{R}^{nd}$.} and underdamped kinetic Langevin system, perturbed by a $\tau_f$-periodic in $t$ force $f(x,t)$:
\begin{equation}\label{Langevin-Eq}
\begin{split}
&dx=vdt,\\
&dv=-\Gamma vdt-\nabla V(x)dt+\varepsilon f(x,t)dt+\sqrt{\mu} \Gamma^{\frac{1}{2}}dW.   
    \end{split}
\end{equation}
Here variables $x,v\in \mathbb{R}^{nd}$ denotes the configuration and velocity of $n$ particles in $\mathbb{R}^d$, respectively, $V: \mathbb{R}^{nd}\to\mathbb{R}$ is the potential, $\Gamma\in \mathbb{R}^{nd\times nd}$ is a symmetric positive definite damping coefficient matrix, and $W$ is an $nd$-dimensional Wiener process. 

In Eq. \cref{Langevin-Eq}, the perturbation parameter $\varepsilon$ controls the intensity of the periodic forcing and is assumed without loss of generality to be positive. We assume $\mu$ is smaller than $\varepsilon$ so that we first have large  deviation principle and then asymptotic analysis of the Maximum Likelihood Path (which will be clarified in the next paragraph). 
We also assume $\varepsilon>0$ is small enough, so that in the absence of noise, i.e. $\mu=0$, there exists at least two stable  periodic states $x_{a}^{\varepsilon}(t)$ and $x_{b}^{\varepsilon}(t)$, separated by an unstable one $x_u^{\varepsilon}(t)$, bifurcated out of two local minima and a saddle of $V(\cdot)$ in the no-forcing case (i.e., $\varepsilon=0$), for which we also assume the saddle is an attractor on the separatrix between the basins of attraction of the two minima. Note the existence of such periodic orbits is guaranteed for small enough $\varepsilon$ due to implicit function theorem (e.g., \cite{Meiss2007}).

It is known that the presence of noise 
(i.e. $\mu\neq0$) introduces a mechanism of transition between periodic solutions of the noiseless system; see e.g., \cite{Freidlin2012,Tao2018, lin2019quasi} for autonomous problems. This article considers how a time-dependent forcing $f(x,t)$ can change the transition rate, which could be understood intuitively as the likelihood of jumping from one basin of attraction to another, and this likelihood is characterized by the transition between the stable periodic states $x_{a}^{\varepsilon}(t)\to x_{b}^{\varepsilon}(t)$, which is impossible without the noise and thus termed as `metastable transition'.  
To quantify such transitions, which involve infinite loopings around $x_{a}^{\varepsilon}(t)$ and $x_{b}^{\varepsilon}(t)$ as $t\to \pm\infty$, it is important to specify the boundary conditions of the transition $x(t)$. Based on knowledge of infinite-time metastable transition in autonomous systems, it might be tempting to consider boundary conditions $\lim_{t\to -\infty} x(t)-x_a^\varepsilon(t) = 0$ and $\lim_{t\to+\infty} x(t)-x_b^\varepsilon(t)=0$, but we will actually allow an additional phase difference,  whose ramification will be detailed later on. 


More precisely, let $\|\cdot\|_B$ denote a weighted norm, $\|x\|_B=\sqrt{x^TBx}$, where $x\in\mathbb{R}^{nd}$ and $B\in \mathbb{R}^{nd\times nd}$ is a positive definite matrix. Given an autonomous problem
\begin{equation}\label{eq:LangevinNoForcing}
\begin{split}
&dx=vdt,\\
&dv=-\Gamma vdt-\nabla V(x)dt+\sqrt{\mu} \Gamma^{\frac{1}{2}}dW,
    \end{split}
\end{equation}
equipped with boundary condition $x(T_1)=x_1$ and $x(T_2)=x_2$, Freidlin-Wentzell large deviation theory gives that, as $\mu\to 0$, the probability density of having a solution $x(\cdot)$ is formally asymptotically equivalent to $\exp\{-S_{T_1, T_2}[x]/\mu\}$, where the associated action functional $S_{T_1,T_2}[x]$ is given by
\begin{equation*}
S_{T_1, T_2}[x]=\left\{
\begin{array}{cc}
 \frac{1}{2}\int_{T_1}^{T_2}\|\ddot{x}+\Gamma\dot{x}+\nabla V(x)\|^2_{\Gamma^{-1}}dt,  & {x\in \bar{C}_{x_1}^{x_2}(T_1,T_2)},  \\
     \infty, & \text{otherwise,}
\end{array} \right.
\end{equation*}
where $\bar{C}_{x_1}^{x_2}(T_1,T_2)$ denotes the space of absolutely continuous functions in $[T_1, T_2]$ that satisfy $x(T_1)=x_1$ and $x(T_2)=x_2$.

The non-autonomy of \eqref{Langevin-Eq} creates extra challenges, but it was established in \cite{Dykman1997Resonant} that
in the $\mu\to0$ limit, the transition rate (this time between metastable periodic orbits in system \cref{Langevin-Eq} instead of metastable fixed points in system \cref{eq:LangevinNoForcing}) 
can be essentially characterised by $\exp(-S^{\varepsilon}/\mu)$, where the quantity $S^{\varepsilon}$ is described as follows: 
  \begin{align}
 &S^{\varepsilon}=\inf_{x\in \bar{C}_{x_1}^{x_2}(\mathbb{R}), \lim_{t\to -\infty} x(t)-x_a^{\varepsilon}(t) = 0, 
 \lim_{t\to +\infty} x(t)-x_b^{\varepsilon}(t) = 0} S^{\varepsilon}[x(t)], \label{Energy-S}\\
 &S^{\varepsilon}[x(t)]=\frac{1}{2}\int_{-\infty}^{+\infty}\|\ddot{x}+\Gamma\dot{x}+\nabla V(x)-\varepsilon f(x,t)\|^2_{\Gamma^{-1}}dt.
 \nonumber
\end{align}
Here the minimization in \cref{Energy-S} is performed in the space of absolutely continuous functions in $\mathbb{R}$.
We remark that the usage of boundary conditions at $\pm\infty$ is reasonable as the minimum in \cref{Energy-S} is generally achieved when $T_2-T_1\to\infty$. In other words, maximum likelihood transition time between two metastable periodic states is infinite (e.g., \cite{Heymann2008}). In the most parts of this article, we will be only concerned with local minimizers of $S^{\varepsilon}[x(t)]$.  The reason is convexity is not guaranteed and global minimization might be too difficult. We call these minimizers maximum likelihood paths (MLPs) throughout this paper (they are also called
instantons in the physical literature). Since we will be considering just local minimizers, assume without loss of generality that there is only one $x_u^\varepsilon(t)$ as we will be just considering transitions through its neighborhood.

In the absence of a non-autonomous forcing ($\varepsilon=0$), hopping between metastable states of kinetic Langevin equation has been studied in detail, including the characterization of MLPs and the corresponding action value. For example, Souza and Tao \cite{Souza2018Metastable} analysed minimisers of the Freidlin–Wentzell action for the kinetic Langevin equation
with respect to various types of friction coefficients for illustrating features in kinetic Langevin metastable transitions that markedly differ from the familiar overdamped picture. The idea studied in this paper is to use specific periodic perturbations ($\varepsilon\neq0$) to facilitate metastable transitions. This idea is inspired, for example, by the investigation of periodically-perturbed Markov jump processes in chemical and 
epidemiological applications \cite{Escudero2008, Assaf2008, Billings2018}, by stochastic resonance \cite{Presilla1989, Jung1991, Gammaitoni1998, Lucarini2019}, by the use of non-gradient forcing (which can be interpreted as an irreversible component, just like how time-dependent perturbation can also be interpreted so) for changing transition rate \cite{Heymann2008Pathways}, and by many successes in controlling deterministic systems using periodic perturbations \cite{Kapitza1951, Paul2010, Rice2000, Shapiro2012, Judson1992, Assion1998, Daniel2003, Brixner2001, Levis2001, Koon2013, Tao2016Temporal, Surappa2018Analysis, Xie2019Parametric}. To quantify how periodic perturbation can change the rate of metastable transition, a key practical question then becomes how to compute the minimum of the Freidlin-Wentzell action functional \cref{Energy-S}.

Continuous and significant efforts have already been made to understand noise induced transitions or escapes in the presence of a periodic driving. Smelyanskiy \emph{et al.} \cite{Smelyanskiy1997} and Dykman \emph{et al.} \cite{Dykman2001} proved that the escape probabilities can be changed very strongly even by a comparatively weak force. Agudov \emph{et al.} \cite{Agudov2001, Dubkov2004}  studied the effect of noise-enhanced stability of periodically driven metastable states. Chen \emph{et al.} \cite{Chen2019} identified a most likely noise-induced transition path under periodic forcing in the framework of finite noise.  However, these works mainly have been focused on first-order systems, many of which may be viewed as the overdamped limits of second-order systems\footnote{We also note overdamped Langevin (without time-dependence) is a reversible Markov process, while kinetic Langevin considered here is irreversible, and its rare event quantification, even without the time-dependent perturbation, can be much more challenging; see e.g., \cite{weinan2004minimum, heymann2008geometric, zhou2008adaptive, wan2011adaptive, lindley2013iterative, grafke2015instanton, grafke2017long, forgoston2018primer, dahiya2018ordered, dahiya2018orderedB, Tao2018, Souza2018Metastable, yang2019computing, cameron2019computing, lin2019quasi, lin2020data}).}. Meanwhile, seminal results on noise activated escape rate of a second-order and under-damped dynamical system exist \cite{Dykman1997Resonant,Maier2001}, although they mostly considered only the case of a single particle and linear additive driving forcing, i.e. $f(x,t)=f(t)$. It is our goal to extend the existent approaches and study the rate of metastable transition in the multi-particle / higher-dimensional systems \cref{Langevin-Eq}, which have potential applications in, for instance, molecules dynamics (an example of healing material defeat, which is important in material sciences, will be provided in \cref{sec:experiments}). We also consider more general nonlinear forcings; of particular interest is when $f$ is a parametric perturbation, e.g. $f(x,t)=A\cos(\omega t+\theta)x$, where parameters $A$, $\omega$, $\theta$ represent the amplitude, frequency, and phase of the forcing respectively. We will show explicitly that a specific choice of $\omega$ will lead to a parametric resonant enhancement of the transition rate. We remark that the terminology `resonance' here means that the quasi-potential / transition rate peaks at specific perturbation frequencies \cite{Dykman1997Resonant,Smelyanskiy1997}, which is different from stochastic resonance \cite{Lucarini2019} whereby the noise can lead to the amplification of the input signal.

To accomplish these goals, our point of departure is a higher-order Euler-Lagrange equation \cite{Riahi1972} characterization of the MLPs associated with 2nd-order SDEs (which can be converted to a 1st-order system, however with degenerate noise). Then we utilize a linear-theory calculation inspired by \cite{Dykman1997Resonant,Maier2001,Assaf2008,Smelyanskiy1997} to approximate the minimum of Freidlin-Wentzell action \cref{Energy-S} to the first order of $\varepsilon$. Specifically, our main contributions are (i) to develop  a higher-order Hamiltonian formalism to reformulate time-dependent Freidlin-Wentzell action functional; (ii) to approximate the hopping rate between metastable periodic orbits of system \cref{Langevin-Eq} in terms of the unperturbed MLP explicitly, based on heteroclinic perturbation analysis; and (iii) to explicitly identify parametric resonant frequency in the context of metastable transition, and uncover the impact of such resonance to the metastable transition rate in systems of practical relevance. 
\par
Several facts have to be mentioned: (i) To quantify the effect of external forcing on transition rate, we need to understand how MLPs in \cref{Langevin-Eq} change under the perturbation. After casting the rare event problem in 2nd-order systems as a Hamiltonian formalism (for this formalism for 1st-order systems, see e.g., \cite{Heymann2008,Khovanov2013,grafke2017long,Chen2019}), we will convert the transition rate quantification problem to the persistence of heteroclinic connections between periodic orbits after a nonautonomous Hamiltonian perturbation (similar problems have been considered in \cite{Escudero2008,Assaf2008}, however only for 1 Degree-Of-Freedom (DOF)). Such persistence is a classical question in dynamical systems and has been answered by Melnikov \cite{Melnikov1963,Guckenheimer1983} for 1 DOF. One key issue with kinetic Langevin systems considered here is, Melnikov's approach is no longer directly applicable even in the single particle situation, because it corresponds to a Hamiltonian system with 4 variables (2 DOF). Unfortunately, Melnikov's method was devised initially to compute the distance between stable and unstable manifolds only for planar Hamiltonian vector fields. Nevertheless, the perturbed heteroclinic connection as intersections of stable/unstable manifolds \cite{Yagasaki2018} can still be investigated. An inspiring article \cite{Capinski2018} for instance extends Melnikov’s method to give a condition under which the stable/unstable manifolds intersect transversely in multidimensional setting. 
In essence, our approach is related to Melnikov's method, but at the same time a generalization as we worked out the 1st-order perturbative expansion in higher-dimensions; 
(ii) Unlike in the autonomous kinetic Langevin case considered in \cite{Souza2018Metastable}, one can no longer relate the perturbed MLPs to a 2nd-order deterministic equation due to the loss of delicate statistical mechanical structure (which is, roughly speaking, a transverse decomposition \cite{bouchet2016generalisation} and consequent detailed balance after momentum reversible); instead, 4nd-order Euler-Lagrange equation is necessary. To analyze it, we adopt a perturbation analysis for approximating the rate change along perturbed MLPs directly, which is independent of specific form of perturbed MLPs. 
(iii) There have been previous studies on linear resonance, but our investigation in parametric resonance is, to the best of our knowledge, new.
 
The paper is organized as follows.  Our general theory is in \cref{sec:main}. After reformulating the variational problem as a higher order Hamiltonian formalism in \cref{sec:2.1}, we obtain an equivalent description of Freidlin-Wentzell action in \cref{sec:2.2} and further characterize the transition rate in \cref{sec:2.3}. We then discuss the characterization of resonant frequency for parametric resonance in \cref{sec: PR}. Numerical experiments of parametric resonance are in \cref{sec:experiments},  and conclusions follow in \cref{sec:conclusions}.

\section{General theory}
(To simplify the notation, throughout the paper, we use symbols without superscript $\varepsilon$ to represent quantities in the case of the unperturbed system, i.e., $\varepsilon=0$).

\label{sec:main}
\subsection{Higher-order Lagrange and Hamiltonian formulation}
\label{sec:2.1}
To better understand the minimizer $x(t)$ (i.e., MLP) of  Freidlin-Wentzell action functional $S^{\varepsilon}[x(t)]$, we start with the Euler-Lagrange equation of the  higher-order variational principle \cite{Riahi1972} (see also \cref{sec:EL}), given by
\begin{equation} \label{EL}
\frac{\delta S^{\varepsilon}[x]}{\delta x}=\frac{\partial\mathcal{L}}{\partial x}-\frac{d}{dt}\frac{\partial\mathcal{L}}{\partial\dot{x}}+\frac{d^2}{dt^2}\frac{\partial\mathcal{L}}{\partial\ddot{x}}=0,
\end{equation}
equipped with boundary conditions
$$\lim_{t\to-\infty}x(t)-x_{a}^{\varepsilon}(t)=0,\; \lim_{t\to+\infty}x(t)-x_{b}^{\varepsilon}(t)=0,$$
where the Lagrangian $\mathcal{L}(t;x,\dot{x},\ddot{x})$ is given by 
\begin{equation}\label{L-special}
\mathcal{L}(t;x,\dot{x},\ddot{x})=\frac{1}{2}\|\ddot{x}+\Gamma\dot{x}+\nabla V(x)-\varepsilon f(x,t)\|^2_{\Gamma^{-1}}.
\end{equation}
Note that this differs from traditional Lagrangian mechanics where $\mathcal{L}$ depends only on $t$, $x$ and $\dot{x}$, and the root of this difference lies in that noise is degenerate if we rewrite the kinetic Langevin equation as a first-order system. Consequently, the Euler-Lagrange equations \cref{EL} is a system of fourth-order differential equations for variable $x(t)$.
Another remark is that although each MLP solves the Euler-Lagrange equation, its solutions are not unique; later on we will establish a family of approximate solutions indexed by a parameter $t_0$.

We now convert this high-order Lagrangian problem to the Hamiltonian picture. For this purpose, let $q_1=x$, $q_2=\dot{x}$ and introduce new variables $p_1$, $p_2$ via
\begin{equation} \label{Generalized momentum}
p_1=\frac{\partial \mathcal{L}}{\partial \dot{x}}-\frac{d}{dt}\frac{\partial \mathcal{L}}{\partial \ddot{x}}, \;\;\; p_2=\frac{\partial \mathcal{L}}{\partial \ddot{x}}
\end{equation}
respectively, which are called generalized momentum of the prescribed system \cite{Riahi1972}. Then $x$ solves Euler-Lagrange equations  \cref{EL} if and only if
\begin{equation} \label{P1}
\dot{p}_1=\frac{dp_1}{dt}=\frac{d}{dt}\frac{\partial \mathcal{L}}{\partial \dot{x}}-\frac{d^2}{dt^2}\frac{\partial \mathcal{L}}{\partial \ddot{x}}=\frac{\partial \mathcal{L}}{\partial x}=\frac{\partial \mathcal{L}}{\partial q_1}.
\end{equation}
Since our Lagrangian $\mathcal{L}(t;x,\dot{x},\ddot{x})$ satisfies the following non-degeneracy condition
$$\det(\frac{\partial^2\mathcal{L}}{\partial \ddot{x}\partial \ddot{x}})_{nd\times nd}\neq 0,$$
it follows from implicit function theorem that $\ddot{x}$ could be locally expressed as a function of $t$, $q_1$, $q_2$, $p_2$; denote it by
  \begin{displaymath}
  \ddot{x}=\mathbf{G}(t;q_1, q_2, p_2).
   \end{displaymath}
 Define the Hamiltonian 
\begin{align}
H:&=H(t,q_1, q_2, p_1, p_2)\notag\\
&=p_1\cdot\dot{x}+p_2\cdot\ddot{x}-\mathcal{L}(t;x,\dot{x},\ddot{x})\notag\\
&=p_1\cdot\dot{q}_1+p_2\cdot \mathbf{G}(t;q_1,q_2, p_2)-\mathcal{L}(t; q_1,q_2,\mathbf{G}(t;q_1, q_2,p_2)),
\end{align}
where $``\cdot"$ represents the scalar product of vectors. It follows that 
\begin{equation} \label{Calculation}
  \begin{split}
  &\frac{\partial H}{\partial q_1}=(D_{q_1}\mathbf{G}(t;q_1,q_2, p_2))^Tp_2-\frac{\partial\mathcal{L}}{\partial q_1}-(D_{q_1}\mathbf{G}(t;q_1,q_2, p_2))^T\frac{\partial\mathcal{L}}{\partial\ddot{x}}=-\frac{\partial\mathcal{L}}{\partial q_1}\\
  &\frac{\partial H}{\partial q_2}=p_1+(D_{q_2}\mathbf{G}(t;q_1,q_2,p_2))^Tp_2-\frac{\partial\mathcal{L}}{\partial q_2}-(D_{q_2}\mathbf{G}(t;q_1,q_2, p_2))^T\frac{\partial\mathcal{L}}{\partial\ddot{x}}=p_1-\frac{\partial\mathcal{L}}{\partial q_2}\\ 
  &\frac{\partial H}{\partial p_1}=\dot{q_1}\\
  &\frac{\partial H}{\partial p_2}= \mathbf{G}(t;q_1, q_2,p_2),
      \end{split}
\end{equation}
where $D_{q_1}\mathbf{G}$ (resp.  $D_{q_2}\mathbf{G}$) denotes the gradient matrix of $\mathbf{G}$ with respect to variable $q_1$ (resp. $q_2$), and superscript `$T$' refers to  transposition.
Hence, the Euler-Lagrange equations \cref{EL} transform into the Hamiltonian differential equations
\begin{equation} \label{Calculation2}
  \begin{split}
  &\dot{q_1}=\frac{\partial H}{\partial p_1},\;\;\;\;\; \dot{q_2}=\frac{\partial H}{\partial p_2},\\
  &\dot{p_1}=-\frac{\partial H}{\partial q_1},\;\;\; \dot{p_2}=-\frac{\partial H}{\partial q_2},  
    \end{split}
\end{equation}
where we use the \cref{P1} and the fact that $\dot{p}_2=p_1-\frac{\partial \mathcal{L}}{\partial \dot{x}}$ by \cref{Generalized momentum}. In fact, the above process can be reversed under the following condition 
$$\det(\frac{\partial^2 H}{\partial p_2 \partial p_2})_{nd\times nd}\neq 0.$$\par
For the specific Lagrangian in \cref{L-special}, simplifications can be made, and the corresponding Euler-Lagrange equation can be written in the equivalent Hamiltonian form as 
\begin{equation} \label{example-H}
  \begin{split}
  &\dot{q}_1=q_2,\;\;\;\;\; \;\;\;\;\;\;\;\;\;\;\;\;\;\;\dot{q}_2=\Gamma p_2-\Gamma q_2-\nabla V(q_1)+\varepsilon f(q_1,t),\\
  &\dot{p}_1=(D_{q_1}\nabla V(q_1))^Tp_2-\varepsilon(D_{q_1}f(q_1,t))^Tp_2,\;\;\; \dot{p}_2=\Gamma p_2-p_1,    \end{split}
\end{equation}
with Hamiltonian function
\begin{align}\label{H-example}
H^{\varepsilon}(t,q_1,q_2,p_1,p_2)
=&p_1^T\dot{q}_1+p_2^T\dot{q}_2-\frac{1}{2}\|\dot{q}_2+\Gamma\dot{q}_1+\nabla V(q_1)-\varepsilon f(q_1,t)\|^2_{\Gamma^{-1}} \notag\\
=&p_1^Tq_2+p_2^T[\Gamma p_2-\Gamma \dot{q_1}-\nabla V(q_1)+\epsilon f(q_1,t)]-\frac{1}{2}p_2^T\Gamma p_2\notag\\
=&\frac{1}{2}p_2^T\Gamma p_2+p_1^Tq_2-p_2^T[\Gamma q_2+\nabla V(q_1)]+\varepsilon p_2^Tf(q_1,t)\\
:=&H_0(q_1,q_2, p_1,p_2)+\varepsilon H_1(t, q_1, q_2, p_1,p_2).
\end{align}

Therefore, each MLP in \cref{Langevin-Eq} corresponds to a solution of Hamilton's equations \cref{example-H} (subject to boundary conditions). When $\varepsilon=0$, the Hamiltonian system admits at least three hyperbolic fixed points $A(q_1=x_a, q_2=p_1=p_2=0)$,  $B(q_1=x_b, q_2=p_1=p_2=0)$ and $O(q_1=x_u, q_2=p_1=p_2=0)$ of the phase space $(q_1,q_2,p_1,p_2)$, originated from the stable fixed points $(x_a,0)$, $(x_b,0)$ and saddle point $(x_u,0)$ of the \cref{Langevin-Eq}, respectively.  
Based on this higher-order Hamiltonian formalism, we now investigate how MLPs in \cref{Langevin-Eq} change under perturbation via understanding how heteroclinic connections in \cref{example-H} change under perturbation. 

Let us first consider the system \cref{Langevin-Eq} in the absence of perturbation, i.e., with $\varepsilon=0$. 
It is known that MLPs among two local minima $x_{a}$, $x_b$ of $V(x)$ correspond to the concatenation between the uphill heteroclinic orbits parametrized by $x_h^{(1)}(t)$ and downhill heteroclinic orbits $x_h^{(2)}(t)$ (e.g., \cite{Freidlin2012, Souza2019}) as long as they exist, which are described by
\begin{align}
&\ddot{x}_h^{(1)}-\Gamma\dot{x}_h^{(1)}+\nabla V(x_h^{(1)})=0,\;\;\; x_h^{(1)}(-\infty)=x_{a},\;\; x_h^{(1)}(+\infty)=x_{u},\label{uphill}\\
&\ddot{x}_h^{(2)}+\Gamma\dot{x}_h^{(2)}+\nabla V(x_h^{(2)})=0,\;\;\; x_h^{(2)}(-\infty)=x_{u},\;\; x_h^{(2)}(+\infty)=x_{b}
\label{downhill}
\end{align}
respectively, where $x_u$ is a saddle of $V(x)$ and it serves as the transition from `uphill' to `downhill'.
The uphill (resp. downhill) heteroclinic orbit will be assumed to exist in this paper (the nonexistence is considered in \cite{Souza2019} and not our focus).

To prepare for later treatments where $\varepsilon$ is no longer zero, we note that $x_h^{(1)}(t-t_0)$ and $x_h^{(2)}(t-t_0)$ are also uphill and downhill heteroclinic orbits for any fixed phase parameter $t_0$. A direct calculation shows they also satisfy \cref{EL} or \cref{example-H} in the case of $\varepsilon=0$, and thus we have the correspondence in  \cref{example-H}, i.e.,
\begin{equation}\label{Relation}
  \begin{split}
  &q_{1,h}^{(1)}(t-t_0)=x_h^{(1)}(t-t_0),\;\;\;\;\;\;\;\; q_{2,h}^{(1)}(t-t_0)=\dot{x}_h^{(1)}(t-t_0),\\
  &p_{2,h}^{(1)}(t-t_0)=2q_{2,h}^{(1)}(t-t_0),\;\;\;\;\; \dot{p}_{1,h}^{(1)}=\left(D_{q_{1,h}^{(1)}}\nabla V\left(q_{1,h}^{(1)}\right)\right)^Tp_{2,h}^{(1)},
  \end{split}
\end{equation}
and 
\begin{equation}\label{Relation2}
  \begin{split}
  &q_{1,h}^{(2)}(t-t_0)=x_h^{(2)}(t-t_0),\;\;\;\;\;\;\;\; q_{2,h}^{(2)}(t-t_0)=\dot{x}_h^{(2)}(t-t_0),\\
  &p_{2,h}^{(2)}(t-t_0)=0,\;\;\;\;\;\;\;\;\;\;\;\;\;\;\;\;\;\;\;\;\;\;\; \dot{p}_{1,h}^{(2)}(t-t_0)=0,
  \end{split}
\end{equation}
respectively. Indeed, these action-minimising trajectories  \cref{Relation} (resp. \cref{Relation2}) are heteroclinic connections among two hyperbolic fixed points $A$ (resp. $O$) and $O$ (resp. $B$) of the forceless ($\varepsilon=0$) Hamiltonian system \cref{example-H}. 
Note the motion from $A$ (resp. $O$) to $O$ (resp. $B$) is the intersection of the unstable  manifold of $A$ (resp. $O$) and the stable manifold of $O$ (resp. $B$) in their respective systems. 

Now we return to the non-autonomously perturbed Hamiltonian system, i.e. \cref{example-H} with $\varepsilon\neq0$. 
Equivalently, we have the suspended system:
\begin{equation} \label{Suspend}
  \begin{split}
  &\dot{q}_1=q_2,\;\;\;\;\; \;\;\;\;\;\;\;\;\;\;\;\;\;\;\dot{q}_2=\Gamma p_2-\Gamma q_2-\nabla V(q_1)+\varepsilon f(q_1,\theta),\\
  &\dot{p}_1=(D_{q_1}\nabla V(q_1))^Tp_2-\varepsilon(D_{q_1}f(q_1,t))^Tp_2,\;\;\; \dot{p}_2=\Gamma p_2-p_1, \\
  &\dot{\theta}=1,\end{split}
\end{equation}
where $(q_1,q_2,p_1,p_2,\theta)\in \mathbb{R}^{4nd}\times\mathcal{S}^1$ ($\mathcal{S}^1=\mathbb{R}/\tau_f$). For $\varepsilon$ sufficiently small, \cref{Suspend} possesses
a Poincar\'{e} map $\mathcal{P}_{\varepsilon}^{t_0}: \sum_{t_0}\to \sum_{t_0}$, 
where $\sum_{t_0}=\{(q_1, q_2, p_1, p_2, \theta)|\theta=t_0\in[0,\tau_f]\}\subset\mathbb{R}^{4nd}\times\mathcal{S}^1$ is the global cross section \footnote{If the orbit of every point $(q_1,q_2,p_1,p_2,t)\in\mathbb{R}^{4nd+1}$ eventually crosses an $4nd$ dimensional surface $\sum_{t_0}$ and then returns to $\sum_{t_0}$ at a later time, then $\sum_{t_0}$ is a global section \cite{Meiss2007}.} at time $t_0$ of the suspended autonomous flow. 

In the perturbed system \cref{Suspend} $\gamma_a=A\times\mathcal{S}^1$, $\gamma_b=B\times\mathcal{S}^1$ and $\gamma_u=O\times\mathcal{S}^1$, as periodic orbits of suspended system with $\varepsilon=0$, are also perturbed. We will denote the perturbed (unique) periodic orbits by $\gamma_a^{\varepsilon}$, $\gamma_b^{\varepsilon}$ and $\gamma_u^{\varepsilon}$, respectively, the first component ($q_1$ component) of which give periodic orbits $x_a^{\varepsilon}(t)$, $x_b^{\varepsilon}(t)$, $x_u^{\varepsilon}(t)$ of noiseless system \cref{Langevin-Eq}. 
For generic Hamiltonians, the existence of the perturbed periodic orbits is guaranteed, by implicit function theorem, at least for sufficiently small $\varepsilon$.
We further assume that the heteroclinic trajectory connecting the unstable manifold of $\gamma^{\varepsilon}_{a}$ (resp. $\gamma^{\varepsilon}_{u}$) and the stable manifold of $\gamma^{\varepsilon}_{u}$ (resp. $\gamma^{\varepsilon}_{b}$) 
based on $\sum_{t_0}$ in system \cref{example-H} survives after the non-autonomous forcing is switched on (i.e., $\varepsilon \neq0$) (see \cref{Assume} for more details). Let us denote the perturbed instanton connection by the pair $q_{1,h}^{(i),\varepsilon }(t;t_0), q_{2,h}^{(i),\varepsilon }(t;t_0),p_{1,h}^{(i),\varepsilon }(t;t_0), p_{2,h}^{(i),\varepsilon }(t;t_0)$, $i=1,2$. Then, by the equivalence of Hamiltonian \cref{example-H}  and Largangian Mechanics \cref{EL}, we also obtain the existence of heteroclinic connection $x_h^{(1),\varepsilon}(t;t_0)$ from $x_{a}^{\varepsilon}(t)$ to $x_{u}^{\varepsilon}(t)$, and heteroclinic orbit $x_h^{(2),\varepsilon}(t;t_0)$ from $x_{u}^{\varepsilon}(t)$ to $x_{b}^{\varepsilon}(t)$ in \cref{EL}. Note that the initial time, $t_0$, appears explicitly, since solutions of the  \cref{EL} are not invariant under arbitrary translations in time (Eq. \cref{EL} is nonautonomous  for $\varepsilon\neq 0$). 

If $\varepsilon=0$, the heteroclinic connections $x_h^{(1),\varepsilon}(t;t_0)$, $x_h^{(2),\varepsilon}(t;t_0)$ degenerate to the uphill and downhill heteroclinic orbit $x_h^{(1)}(t-t_0)$, $x_h^{(2)}(t-t_0)$, respectively. Different from the autonomous case described above (where the unperturbed action is invariant to the specific choice of $t_0$, e.g., \cite{Tao2018, Souza2019}), different $t_0$ gives different local minimum as the concatenation of $x_h^{(1),\varepsilon}(t;t_0)$ and $x_h^{(2),\varepsilon}(t;t_0)$ (for each fixed $t_0$ value) gives a critical point of the action for system \cref{Langevin-Eq}.
The reason of this difference lies in that the energy in \cref{example-H} is not conserved, and the action along the perturbed instanton now includes an integral of $H^{\varepsilon}$ over time (see \cref{sec:2.2} for the equivalence of Fredlin-Wentzell action functional and Hamiltonian action). 
Consequently, we optimize over $t_0$ to obtain an MLP and the optimal transition rate of system \cref{Langevin-Eq}. 
Thus \cref{Energy-S} can be formally rewritten as follows \cite{tao2019simply}:
\begin{equation}\label{TwoS}
 S^{\varepsilon}=\min_{t_0}\{ S^{\varepsilon}[x_h^{(1),\varepsilon}(t;t_0)]+S^{\varepsilon}[x_h^{(2),\varepsilon}(t;t_0)]\}.
\end{equation}


To compute \cref{TwoS}, we will focus on the calculation of $S^{\varepsilon}[x_h^{(1),\varepsilon}(t;t_0)]$, since 
$S^{\varepsilon}[x_h^{(2),\varepsilon}(t;t_0)]$ will be, to the first order of $\varepsilon$, $0$. This is because $x_h^{(2),\varepsilon}(t;t_0)$ is moving along the perturbed downhill heteroclinic orbit as we will show in \cref{sec:2.3} (similar result originated from autonomous kinetic Langevin has been verified in \cite{Souza2019}).
From a physical point of view, the reason for such a result is that $x_h^{(2),\varepsilon}(t;t_0)$ is a relaxation trajectory, i.e. once the system has approached the vicinity of an unstable periodic state $x_{u}^{\varepsilon}(t)$, it will eventually be attracted to another stable periodic state $x_{b}^{\varepsilon}(t)$ with a probability $\sim1/N$ (which goes into the prefactor; $N$ is the number of attraction basins whose boundaries include $x_u^\varepsilon$), without requiring extra noise. Thus, in order to quantify the metastable transition rate between $x_{a}^{\varepsilon}(t)$ and 
$x_{b}^{\varepsilon}(t)$ in this case, we convert it to an escape problem from the vicinity of the stable periodic state $x_{a}^{\varepsilon}(t)$ to that of unstable one $x_{u}^{\varepsilon}(t)$. 

\begin{remark}\label{Assume}
The perturbed system \cref{example-H} will not, in general, maintain the intersection between  
the unstable and stable manifolds of $A$ (resp. $O$) and $O$ (resp. $B$) \cite{Escudero2008}: these manifolds might intersect, preserving the existence of the heteroclinic connection, but they also might not, in which case the heteroclinic is destroyed. Interestingly, Capinski \cite{Capinski2018} proposed a  Melnikov type approach for establishing transversal intersections of stable/unstable manifolds in multidimensional setting. This result could ensure the existence of heteroclinic connections in system \cref{example-H} if appropriate conditions are imposed. Nevertheless, for simplicity we just assume $\varepsilon$ is small enough so that a heteroclinic connection in the perturbed system exists.
\end{remark}

\subsection{Reformulation of Fredlin-Wentzell action functional}
\label{sec:2.2}
As described above, to obtain the minimum of the Freidlin-Wentzell action functional, the core of this calculation is the transition from the stable periodic orbit $x_{a}^{\varepsilon}(t)$ to the unstable one $x_{u}^{\varepsilon}(t)$. And our strategy will be to calculate the action along the perturbed instanton, which is the heteroclinic connection of the perturbed, time-dependent Hamiltonian system \cref{example-H}. More concretely, we start by reformulating the action functional  $S^{\varepsilon}[\cdot]$ in \cref{Energy-S} in a form convenient in this subsection, and then conduct a derivation on correction of the action in subsequently \cref{sec:2.3}. For simplicity of exposition, we drop the superscript `$(1)$' in notation $x_h^{(1)}(t)$, $x_h^{(1).\varepsilon}(t;t_0)$. 

Let us now consider a Hamiltonian action
\begin{align}\label{H-action}
A^{\varepsilon}[\gamma]=&\int_{\gamma}p_1^Tdq_1+p_2^Tdq_2-H^{\varepsilon}dt\notag\\
=&\int_{-\infty}^{\infty}p_1^T(t)\frac{dq_1}{dt}+p_2^T(t)\frac{dq_2}{dt}-H^{\varepsilon}(t,q_1(t),q_2(t), p_1(t),p_2(t))dt,
\end{align} 
where $\gamma$ is a path in phase space. To be more precise, $\gamma:(-\infty, \infty)\to \mathbb{R}^{4nd}$ and is denoted by $\gamma=\{(q_1^T(t), q_2^T(t), p_1^T(t), p_2^T(t)),-\infty < t< \infty\}$.  
$H^{\varepsilon}$ takes the form of \cref{H-example}. We remark that variables $q_1$, $q_2$, $p_1$, $p_2$ in \cref{H-action} are mutually independent. It is well known (e.g., \cite{Meiss2007}) that the curve of stationary action \cref{H-action} is the Hamiltonian trajectory \cref{example-H}. We further have the following result.

\begin{proposition}\label{Two-action}
$S^{\varepsilon}[x]=A^{\varepsilon}[\gamma]$ along the instanton solutions of \cref{EL} or  \cref{example-H}. That is
\begin{equation}\label{LH}
S^{\varepsilon}\left[x_h^{\varepsilon}(t;t_0)\right]=A^{\varepsilon}\left[\left(q_{1,h}^{\varepsilon }(t;t_0), q_{2,h}^{\varepsilon }(t;t_0),p_{1,h}^{\varepsilon }(t;t_0), p_{2,h}^{\varepsilon }(t;t_0)\right)\right],
\end{equation} 
which is dependent of $t_0$ with $t_0\in[0, \tau_{f}]$ when $\varepsilon\neq 0$. Here $x_h^{\varepsilon}(t;t_0)=q_{1,h}^{\varepsilon }(t;t_0)$ is heteroclinic connection connecting $x_{a}^{\varepsilon}(t)$ and $x_{u}^{\varepsilon}(t)$ in system \cref{EL}, and a relation between $x_h^{\varepsilon}(t;t_0)$ and $q_{2,h}^{\varepsilon }(t;t_0)$, $p_{1,h}^{\varepsilon }(t;t_0)$, $p_{2,h}^{\varepsilon }(t;t_0)$ is given by \cref{Generalized momentum}.
\end{proposition}

\cref{Two-action} holds as a result of equivalence of Hamiltonian and Lagrangian Mechanics under Legendre condition. For more details on proof of \cref{Two-action}, the reader is referred to \cite{Meiss2007}.

Served for the next subsection, we further calculate the minimum of the Freidlin-Wentzell action functional for system \cref{Langevin-Eq} in the absence of $\varepsilon$. In view of \cref{H-action}, \cref{H-example}, \cref{uphill}, we get
\begin{equation*}
\begin{array}{lll}
A_0&=&\int_{-\infty}^{\infty}[p_{1,h}^T\dot{q}_{1,h}+p_{2,h}^T\dot{q}_{2,h}-H_0(q_{1,h},q_{2,h},p_{1,h},p_{2,h})]dt\\
&=&\int_{-\infty}^{\infty}p_{1,h}^Tq_{2,h}+2q_{2,h}^T\dot{q}_{2,h}-p_{1,h}^Tq_{2,h}+2\dot{q}_{1,h}^T\nabla V(q_{1,h})dt\\
&=&2[V(x_{u})-V(x_{a})],
\end{array}
\end{equation*}
\begin{equation*}
\begin{array}{lll}
S_0=S[x_h]&=&\frac{1}{2}\int_{-\infty}^{\infty}\|\ddot{x}_h+\Gamma\dot{x}_h+\nabla V(x_h)\|^2_{\Gamma^{-1}}dt\\
&=&\frac{1}{2}\int_{-\infty}^{\infty}\|\ddot{x}_h-\Gamma\dot{x}_h+\nabla V(x_h)\|^2_{\Gamma^{-1}}+4\dot{x}_{h}^{T}(\ddot{x}_h+\nabla V(x_h))dt\\
&=&2[V(x_{u})-V(x_{a})],
\end{array}
\end{equation*}
if we choose homogeneous velocity boundary conditions, $\dot{x}_h(-\infty)=0$, $\dot{x}_h(-\infty)=0$,
where $x_h(t)$ is uphill heteroclinic orbit given in \cref{uphill}. Again we get $S_0=A_0.$

\begin{remark}\label{xutoxb}
\cref{Two-action} also holds for the case of $x_u^{\varepsilon}$(t) to $x_b^{\varepsilon}(t)$ transition, with the corresponding value of action $S_0$ being $0$ \cite{Souza2019} when $\varepsilon=0$ (Because downhill heteroclinic orbit $x_h^{(2)}(t)$ by definition \cref{downhill} is the zero of $\frac{1}{2}\int_{-\infty}^{+\infty}\|\ddot{x}+\Gamma\dot{x}+\nabla V(x)\|^2_{\Gamma^{-1}}dt$). 
\end{remark}

\subsection{Relating the minimizers of $S^{\varepsilon}$ and $S$}
\label{sec:2.3}
In the following, we will use the reformulation \cref{LH} to study the relationship between $S^{\varepsilon}[\cdot]$ and $S[\cdot]$. To do so, it is equivalent to deal with the relationship between $A^{\varepsilon}[\cdot]$ and $A[\cdot]$. 

In light of \cref{Two-action}, we provide a linear-theory calculation of the action $S^{\varepsilon}$ inspired by \cite{Assaf2008} and approximate the rate of the metastable transition. 
Assume $\varepsilon\ll 1$ so that the term $H_1(t,q_1,q_2,p_1,p_2)$ in \cref{H-example} can be treated perturbatively. Let us expand the perturbed instanton of $H^{\varepsilon}(t,q_{1,h}^{\varepsilon}(t;t_0),q_{2,h}^{\varepsilon}(t;t_0), p_{1,h}^{\varepsilon}(t;t_0),p_{2,h}^{\varepsilon}(t;t_0))$ to the first order in $\varepsilon$:
\begin{equation} \label{Perturbed-instanton}
  \begin{split}
  &q_{1,h}^{\varepsilon}(t;t_0)=q_{1,h}(t-t_0)+\varepsilon Q_{1,h}(t;t_0)+\mathcal{O}(\varepsilon^2),\\
  &q_{2,h}^{\varepsilon}(t;t_0)=q_{2,h}(t-t_0)+\varepsilon Q_{2,h}(t;t_0)+\mathcal{O}(\varepsilon^2),\\ 
  &p_{1,h}^{\varepsilon}(t;t_0)=p_{1,h}(t-t_0)+\varepsilon P_{1,h}(t;t_0)+\mathcal{O}(\varepsilon^2),\\
  &p_{2,h}^{\varepsilon}(t;t_0)=p_{2,h}(t-t_0)+\varepsilon P_{2,h}(t;t_0)+\mathcal{O}(\varepsilon^2),
   \end{split}
\end{equation}
where $q_{1,h}(t-t_0)=x_h(t-t_0)$ satisfies uphill dynamics \cref{uphill}, and it's combined with $q_{2,h}(t-t_0)$, $p_{1,h}(t-t_0)$, $p_{2,h}(t-t_0)$ to stand for the (known) instanton solution of \cref{example-H} in the absence of $\varepsilon$.
To calculate the action $A^{\varepsilon}[\cdot]$ given in \cref{H-action}, we expand the integrand in $\varepsilon$ by a Taylor theorem and obtain, in the first order,
\begin{align*}
(&p_{1,h}+\varepsilon P_{1,h})^T(\dot{q}_{1,h}+\varepsilon \dot{Q}_{1,h})+(p_{2,h}+\varepsilon P_{2,h})^T(\dot{q}_{2,h}+\varepsilon \dot{Q}_{2,h})\notag\\
&-H_0(q_{1,h},q_{2,h},p_{1,h},p_{2,h})-\varepsilon Q_{1,h}^T\frac{\partial H_0}{\partial q_{1,h}}-\varepsilon Q_{2,h}^T\frac{\partial H_0}{\partial q_{2,h}}-\varepsilon P_{1,h}^T\frac{\partial H_0}{\partial p_{1,h}}-\varepsilon P_{2,h}^T\frac{\partial H_0}{\partial p_{2,h}}\notag\\
&-\varepsilon H_1(q_{1,h},q_{2,h},p_{1,h},p_{2,h})\notag\\
\simeq&p_{1,h}^T\dot{q}_{1,h}+p_{2,h}^T\dot{q}_{2,h}-H_0(q_{1,h},q_{2,h},p_{1,h},p_{2,h})+\varepsilon p_{1,h}^T\dot{Q}_{1,h}+\varepsilon P_{1,h}^T\dot{q}_{1,h}+\varepsilon p_{2,h}^T\dot{Q}_{2,h}\notag\\
&+\varepsilon P_{2,h}^T\dot{q}_{2,h}+\varepsilon Q_{1,h}^T\dot{p}_{1,h}+\varepsilon Q_{2,h}^T\dot{p}_{2,h}-\varepsilon P_{1,h}^T\dot{q}_{1,h}-\varepsilon P_{2,h}^T\dot{q}_{2,h}\notag\\
&-\varepsilon H_1(q_{1,h},q_{2,h},p_{1,h},p_{2,h})\notag\\
=&p_{1,h}^T\dot{q}_{1,h}+p_{2,h}^T\dot{q}_{2,h}-H_0(q_{1,h},q_{2,h},p_{1,h},p_{2,h})+\varepsilon p_{1,h}^T\dot{Q}_{1,h}+\varepsilon Q_{1,h}^T\dot{p}_{1,h}+\varepsilon p_{2,h}^T\dot{Q}_{2,h}\notag\\
&+\varepsilon Q_{2,h}^T\dot{p}_{2,h}-\varepsilon H_1(q_{1,h},q_{2,h},p_{1,h},p_{2,h}).
\end{align*}
After the integration, the first three terms yield the unperturbed action $S_0$. And the fourth term and sixth term cancel out with the fifth term and seventh term, respectively, via integration by parts. The result is $S^{\varepsilon}(t_0)=S_0+\varepsilon\delta S(t_0)$, where
\begin{align}
\delta S(t_0)&=-\int_{-\infty}^{+\infty}H_1(t, q_{1,h}(t-t_0),q_{2,h}(t-t_0),p_{1,h}(t-t_0),p_{2,h}(t-t_0)dt\notag\\
&=-\int_{-\infty}^{+\infty}p_{2,h}(t-t_0)^Tf(q_{1,h}(t-t_0),t)dt\label{deltaaction}\\
&=-2\int_{-\infty}^{+\infty}\dot{x}_{h}^T(t-t_0)f(x_{h}(t-t_0),t)dt.\label{action-R}
\end{align}
The last step is based on relation \cref{Relation}. To find the optimal 1st-order correction to the minimum action's value, we have to minimize $S^{\varepsilon}(t_0)$ with respect to $t_0$, which thus yields the equation for optimal $t_0$:
\begin{equation*}
 \int_{-\infty}^{+\infty}\left[\dot{x}_{h}^T(t-t_0)\nabla f(x_{h}(t-t_0),t)\dot{x}_{h}(t-t_0)+\ddot{x}^T_{h}(t-t_0)f(x_{h}(t-t_0),t)\right]dt=0.   
\end{equation*}
We now summarize the above results in a concise form as follows:
\begin{theorem}\label{action-change}
Consider non-autonomous kinetic Langevin system \cref{Langevin-Eq}. Assume a heteroclinic connection from $x_{a}$ to $x_{u}$ exists in the noiseless ($\mu=0$) and forceless ($\varepsilon=0$) backward in time system, and $\varepsilon$ is small enough such that a heteroclinic connection from $x_{a}^\varepsilon(t)$ to $x_{u}^\varepsilon(t)$ exists in Euler-Lagrange equation \cref{EL}. Then the escape rate from stable periodic orbit $x_{a}^{\varepsilon}(t)$ to unstable (hyperbolic) periodic orbit $x_u^{\varepsilon}(t)$ is asymptotically equivalent to $\exp(-S^{\varepsilon}/\mu)$, where $S^{\varepsilon}=2[V(x_{u})-V(x_{a})]+\varepsilon \delta S_e + \mathcal{O}(\varepsilon^2)$, and $\delta S_e$ characterizes the leading order effect of external driving on metastable transition. $\delta S_e$ is given by
\begin{equation}\label{S}
\delta S_e=\min_{t_0}\delta S(t_0),\;\;\; \delta S(t_0)=-2\int_{-\infty}^{+\infty}\dot{x}_{h}^T(t-t_0)f(x_h(t-t_0),t)dt, 
\end{equation}
where $x_h(t)$ satisfies equation 
\begin{equation}\label{uphill2}
\ddot{x}_{h}-\Gamma\dot{x}_{h}+\nabla V(x_{h})=0, \;\;\; x_{h}(-\infty)=x_{a},\;\; x_{h}(+\infty)=x_{u}.
\end{equation}
Here, $x_{a}$, $x_u$ are local minimum and saddle point of potential $V(x)$, respectively. 
\end{theorem}

\begin{remark}\label{R1}
\cref{action-change} is consistent with the results of \cite{Dykman1997Resonant}  which considered the $n=d=1$, $\Gamma=\gamma$ and $f(x,t)=f(t)$ case. 
Different from \cite{Dykman1997Resonant}, our method is also applicable to high-dimensional Langevin systems with general (time-dependent, nonlinear) perturbation, while \cite{Dykman1997Resonant} focused on $1$-dimensional Langevin equations with additive periodic perturbation based on the idea of path integral.
\end{remark}

\begin{remark}\label{R2}
For the general case where $\Gamma$ depends on position and velocity, i.e. $\Gamma(x,\dot{x})
$, \cref{action-change} still holds with slight modification i.e. replacing $\Gamma$ by $\Gamma(x,\dot{x})$.
\end{remark}

A similar procedure can be used to understand the $x_{u}^{\varepsilon}(t)$ to $x_{b}^{\varepsilon}(t)$  transition, and $\delta S(t_0)$ given in \cref{deltaaction} will vanish in this case due to 
Eq. \cref{Relation2}. Together with 
\cref{xutoxb}, we verify that $S^{\varepsilon}[x_h^{(2),\varepsilon}(t;t_0)]$ in \cref{TwoS} is zero, to the first order of $\varepsilon$. Consequently,  combining with \cref{action-change}, a 
 natural result on metastable transition rate is that:
 
 \begin{theorem}\label{Thm2}
Under the same conditions as stated in \cref{action-change}, further assume that $x_u$ is the only attractor on the separatrix between the basins of attraction of $x_a$ and $x_b$, a heteroclinic connection from $x_{u}$ to $x_{b}$ exists in the noiseless ($\mu=0$) and forceless ($\varepsilon=0$) system, and $\varepsilon$ is small enough such that a heteroclinic connection from $x_{u}^\varepsilon(t)$ to $x_{b}^\varepsilon(t)$ exists in Euler-Lagrange \cref{EL}. Then the transition rate from stable periodic orbit $x_{a}^{\varepsilon}(t)$ to another stable periodic orbit $x_b^{\varepsilon}(t)$ is asymptotically equivalent to $\exp(-S^{\varepsilon}/\mu)$, where $S^{\varepsilon}=2[V(x_{u})-V(x_{a})]+\varepsilon \delta S_e + \mathcal{O}(\varepsilon^2)$, and $\delta S_e$ is described in \cref{action-change}.
 \end{theorem}


\begin{example}\label{R3}

Let us consider two special forms of forcing $f(x,t)$, the first being the linear forcing already considered in the literature, and the second being a parametric forcing (see e.g., \cite{Paul2010, Koon2013, Tao2016Temporal, Surappa2018Analysis, Xie2019Parametric} for existing applications to deterministic systems). They will be used later to demonstrate the resonant enhancement of transition rate. For a simple illustration, each component of $f(x,t)$ is chosen to be of the same type (although one can treat arbitrary components of general forcing $f(x,t)$ simply by substituting it into the expression \cref{S} in \cref{action-change}, and calculate integral $\delta S(t_0)$ in \cref{S}).
\par
\paragraph{(i)} For a sinusoidal field $(f(t))_j=A_j\cos(\omega_j t+\theta_j),
j=1,...,nd$,
where parameters $A_j$, $\omega_j$, $\theta_j$ represent the amplitude, frequency, phase of $j$th component of forcing respectively, the correction  $\delta S_e$ becomes 
\begin{equation*}
 \delta S_e=-2\min_{t_0}\delta S(t_0)=-2\min_{t_0}\left\{\sum_{j=1}^{nd} \cos(\omega_jt_0+\theta_j+\phi_j)|\int_{-\infty}^\infty A_j\dot{x}_h^j(t) e^{i\omega_j t} dt|\right\}. 
\end{equation*}
Here, $\dot{x}_{h}^{j}(t)$ denotes the $j$th component of $x_h(t)$, the calculation of $\delta S(t_0)$ is as follows:
{
\small
\begin{align*}
&\delta S(t_0)\notag\\
=&-2\int_{-\infty}^{+\infty}[\sum_{j=1}^{nd}A_j\dot{x}_{h}^{j}(t)\cos(\omega_j (t+t_0)+\theta_j)]dt\notag\\
=&\sum_{j=1}^{nd}\left\{-e^{i(\omega_j t_0+\theta_j)}\int_{-\infty}^{+\infty}A_j\dot{x}_{h}^j(t)e^{i\omega_j t}dt+e^{-i(\omega_j t_0+\theta_j)}\int_{-\infty}^{+\infty}A_j\dot{x}_{h}^j(t)e^{-i\omega_j t}dt\right\}\notag\\
=&-2\sum_{j=1}^{nd}[\cos(\omega_j t_0+\theta_j)\Re(\int_{-\infty}^{+\infty}A_j\dot{x}_{h}^j(t)e^{i\omega_j t}dt)-\sin(\omega_j t_0+\theta_j)\Im (A_j\int_{-\infty}^{+\infty}\dot{x}_{h}^j(t)e^{i\omega_j t}dt)]\notag\\
=&-2\sum_{j=1}^{nd}\cos(\omega_j t_0+\theta_j+\phi_j)\left|A_j\int_{-\infty}^{\infty}\dot{x}_{h}^j(t)e^{i\omega_j t}dt\right|, 
\end{align*}
}
where $\sin\phi_j=\frac{\Im\big{(} \sum_{j=1}^{nd}\int_{-\infty}^{\infty}\dot{x}_{h}^j(t)e^{i\omega_j t}dt\big{)}}{|\sum_{j=1}^{nd}\int_{-\infty}^{\infty}\dot{x}_{h}^j(t)e^{i\omega_j t}dt|}.$ Note that for a special homogeneous case where $\omega_j=\omega$, $\theta_j=\theta$ for $j=1,...,nd$, the optimal $t_0$ can be determined explicitly and thus simplification occurs. Specifically, $\delta S_e=-2\left|\sum_{j=1}^{nd}A_j\int_{-\infty}^{\infty}\dot{x}_{h}^{j}(t)e^{i\omega t}dt\right|.$ 
This is because
\begin{align*}
\delta S(t_0)=&-2\int_{-\infty}^{+\infty}[\sum_{j=1}^{nd}A_j\dot{x}_{h}^{j}(t)]\cos(\omega (t+t_0)+\theta)dt\notag\\
=&-e^{i(\omega t_0+\theta)}\int_{-\infty}^{+\infty}\sum_{j=1}^{nd}A_j\dot{x}_{h}^j(t)e^{i\omega t}dt+e^{-i(\omega t_0+\theta)}\int_{-\infty}^{+\infty}\sum_{j=1}^{nd}A_j\dot{x}_{h}^j(t)e^{-i\omega t}dt\notag\\
=&-2\cos(\omega t_0+\theta+\phi)\left|\sum_{j=1}^{nd}A_j\int_{-\infty}^{\infty}\dot{x}_{h}^j(t)e^{i\omega t}dt\right|,
\end{align*}
where $\sin\phi=\frac{\Im\big{(} \sum_{j=1}^{nd}\int_{-\infty}^{\infty}\dot{x}_{h}^j(t)e^{i\omega t}dt\big{)}}{|\sum_{j=1}^{nd}\int_{-\infty}^{\infty}\dot{x}_{h}^j(t)e^{i\omega t}dt|}.$ Thus $\delta S_e=-2|\sum_{j=1}^{nd}A_j\int_{-\infty}^{\infty}\dot{x}_{h}^j(t)e^{i\omega t}dt|.$
We see that the initial phase $\theta$ 
doesn't affect $\delta S_e$. 

\par
\paragraph{(ii)} For forcing in the form of $(f(x,t))_j=A_j\cos(\omega_j t+\theta_j)x_j, j=1,...,nd$,
the corrections $\delta S_e$ takes form of 
\begin{equation*}
 \delta S_e=-2\min_{t_0}\delta S(t_0)=-2\min_{t_0}\{\sum_{j=1}^{nd} \cos(\omega_jt_0+\theta_j+\phi_j)|\int_{-\infty}^\infty A_j\dot{x}_h^j(t)x_h^j(t) e^{i\omega_j t} dt|\},   
\end{equation*}
where $\sin\phi_j=\frac{\Im\big{(} \sum_{j=1}^{nd}\int_{-\infty}^{\infty}\dot{x}_{h}^j(t)x_h^j(t) e^{i\omega_j t}dt\big{)}}{|\sum_{j=1}^{nd}\int_{-\infty}^{\infty}\dot{x}_{h}^j(t)x_h^j(t)e^{i\omega_j t}dt|}.$
 Again, $\delta S_e$ takes a more simple form, when homogeneous case where $\omega_j=\omega$, $\theta_j=\theta$ for $j=1,...,nd$ is considered. That is $\delta S_e=-2|\sum_{j=1}^{nd}A_j\int_{-\infty}^{\infty}\dot{x}_{h}^j(t)x_h^j(t)e^{i\omega_j t}dt|$. This is because
\begin{align*}
\delta S(t_0)
=&-2\int_{-\infty}^{+\infty}[\sum_{j=1}^{nd}A_j\dot{x}_{h}^{j}(t)x_h^j(t)]\cos(\omega (t+t_0)+\theta)dt\notag\\
=&-2\cos(\omega t_0+\theta+\phi)\left|\sum_{j=1}^{nd}A_j\int_{-\infty}^{\infty}\dot{x}_{h}^j(t)x_h^j(t)e^{i\omega t}dt\right|,
\end{align*}
where the last step is based on Euler formula and operation of complex numbers, $\sin\phi=\frac{\Im\big{(} \sum_{j=1}^{nd}\int_{-\infty}^{\infty}\dot{x}_{h}^j(t)x_h^j(t) e^{i\omega t}dt\big{)}}{|\sum_{j=1}^{nd}\int_{-\infty}^{\infty}\dot{x}_{h}^j(t)x_h^j(t)e^{i\omega t}dt|}.$ Hence
$\delta S_e=-2|\sum_{j=1}^{nd}A_j\int_{-\infty}^{\infty}\dot{x}_{h}^j(t)x_h^j(t)e^{i\omega t}dt|.$


\par
In both cases, $\delta S_e$ is dependent on the input frequencies $\omega_j$. In certain applications such as molecular systems, one may not have enough resolution to force each degree of freedom at a different frequency, and we thus can consider for simplicity a special homogeneous case where $\omega_1=\cdots=\omega_n=\omega$, $\theta_1=\cdots=\theta_n=\theta$. Even in this case, there exist special value(s) of $\omega$ that make the action $\delta S_e$ vary greatly as we will show in \cref{sec: PR}, and each such $\omega$  will be called a resonant frequency.
Consistent with physical intuitions, resonant frequencies are related to the  intrinsic frequencies of the unperturbed system \cref{uphill2}. Thanks to \cref{action-change}, we will show that if the heteroclinic connection of the unperturbed system \cref{uphill2} can be found, we can determine the resonant frequencies,  
without any rare event simulation which is computationally very costly, no matter that how high-dimensional and how nonlinear the original system \cref{Langevin-Eq} is.
\end{example}


\section{Parameteric resonance: Characterization of the resonant frequency}\label{sec: PR}  As the general effect of time-dependent forcing on metastable transition has been discussed in the previous section, we now move on to focus on specific forcings. One observation is that when the forcing takes the form of $f(x,t)=A\cos(\omega t+\theta)x$, a resonance-like mechanism will prevail, namely that there exists a special input frequency that leads to a significantly stronger reduction of quasi-potential (and hence enhanced transition rate). This phenomenon is referred to as parametric resonance\footnote{We call it parametric resonance because the forcing is a parametric perturbation. If $f(x,t)=A\cos(\omega t+\theta)$ instead, a similar phenomenon will be called linear resonance.}. 

We will use the tool of stationary phase method to estimate the resonant frequency of parametric resonance based on \cref{R3}. For the sake of simplicity, we will detail the method on single particle cases, i.e. $n=1,d=1,\Gamma=\gamma$. Then we will outline the idea of a generalization to higher dimensions
in \cref{FreH}.


\paragraph{The heteroclinic of the forceless deterministic system}
To understand why there is a resonant frequency and what it is, we will utilize the assumption that the system is underdamped (i.e., $\gamma$ small) and perform asymptotic estimations (approximation sign `$\approx$' in the following presentation means equal sign `$=$' in the $\gamma\to 0$ limit). More precisely, letting $q=x$, $p=\dot{x}$, Hamiltonian $H(q,p)=p^2/2+V(q)$, and energy $E(t)=H(q(t),p(t))$, the forceless heteroclinic connection corresponds to fast oscillation along the Hamiltonian level set and slow change of the energy value (as the heteroclinic goes up-hill). Following \cite{Dykman1997Resonant}, we express the configuration and velocity variables of the heterclinic connection by: 
\begin{equation}\label{AIE}
\begin{split}
  &q(t)=\sum_{n=-\infty}^{+\infty}q_n(E(t))\exp(-in\varphi(t)),\\
  &p(t)=\dot{q}(t)=\sum_{n=-\infty}^{+\infty}p_n(E(t))\exp(-in\varphi(t)), \\
  & \dot{E}(t) \approx \gamma\omega(E)I(E), \qquad \dot{\varphi}=\omega(E(t)),
  \end{split}
\end{equation}
where $I$ and $\varphi$ are action and angle variables, $\omega(E(t))$ is the frequency of oscillation in Hamiltonian system $H$ at energy $E(t)$, $q_n(E(t))$, $p_n(E(t))$ are respectively the amplitude of the $n$th overtone of the configuration $q(t)$ and the momentum $p(t)$, and the last line of \cref{AIE} is obtained via averaging $\dot{E}=\gamma p^2$ over the oscillations.


\paragraph{Parametric resonant frequency}
For the case of parametric perturbation $f(q,t)=A\cos(\Omega t+\theta)q$, by applying Theorem \ref{Thm2} and \cref{R3}, the change of transition rate is characterized by:
\begin{equation}\label{ExDS}
\delta S_e=-2A \left|\int_{-\infty}^{\infty}\dot{q}_{h}(t)q_h(t)e^{i\Omega t}dt\right|.
\end{equation}
For convenience, denote $|\int_{-\infty}^{\infty}\dot{q}_{h}(t)q_h(t)e^{i\Omega t}dt|$ by $|I(\Omega)|$. 
Substitution of \cref{AIE} into \cref{ExDS} gives
\begin{align}\label{thcase1}
I(\Omega)
&=\int_{-\infty}^{+\infty}\left(\sum_{n=-\infty}^{+\infty}p_n[E(t)]e^{-in\varphi(t)}\right)\left(\sum_{m=-\infty}^{+\infty}q_m[E(t)]e^{-im\varphi(t)}\right)e^{i\Omega t}dt\notag\\
&=\int_{-\infty}^{+\infty}\sum_{l=-\infty}^{+\infty}e^{-il\varphi(t)}\left(\sum_{m=-\infty}^{+\infty}p_{l-m}[E(t)]q_m[E(t)]\right)e^{i\Omega t}dt\notag\\
&=\frac{1}{\gamma}\int_{-\infty}^{+\infty}\sum_{l=-\infty}^{+\infty}a_l[E(\tau)]\exp\left[i\frac{1}{\gamma}\big(\Omega \tau-l\psi(\tau)\big)\right]d\tau.
\end{align}
Here we use $a_l$ to denote $\sum_{m=-\infty}^{+\infty}p_{l-m}[E(t)]q_m[E(t)]$ for simplicity. Note that the last step in \cref{thcase1} is based on the change of variable $t=\frac{1}{\gamma}\tau$, because of which we further obtain that $\frac{d\varphi}{d\tau}=\frac{1}{\gamma}\omega(E)$, $\frac{dE}{d\tau}\approx \omega(E)I(E)$, and the function $\psi(\tau)=\gamma\varphi(\tau)$ satisfies $\frac{d\psi}{d\tau}=\omega(E)$. 

For smooth heteroclinic induced by smooth potential $V$, $a_l$ decays exponentially and the integral and infinite sum can be exchanged. We thus consider each term and denote it by
\begin{equation}
I_l(\Omega)=\frac{1}{\gamma}\int_{-\infty}^{+\infty}a_l[E(\tau)]\exp\left[i\frac{1}{\gamma}\big(\Omega \tau-l\psi(\tau)\big)\right]d\tau.
\label{eq:qtlkqjnbgtoilboli23b4124}
\end{equation}
Also because $a_l$ typically decays very fast as $|l|$ increases, in the presentation below we focus on $\Omega$ value that resonates with the dominating mode, denoted also by $l$ under slightly abused notation (unless confusion arises, which will be clarified then). Usually the fundamental frequency is the dominating mode, i.e., $l=1$. In less common cases when $a_1=a_{-1}=0$ (e.g., if $\ddot{q}_h+\omega^2(1+\epsilon \cos(\Omega t))q_h=0$ then this happens and the resonant frequency is actually $\Omega=2\omega$ instead of $\Omega=\omega$, i.e. $l=2$ instead; this toy $q_h$ is not a heteroclinic connection though) or when multiple overtones have comparable amplitudes, there is one resonant frequency associated with each dominating mode.

Away from a stationary phase, the integrand in \cref{eq:qtlkqjnbgtoilboli23b4124} has a slowly varying amplitude but a fast oscillating phase, and it thus mostly cancels out after the integration. If $\tau$ has a stationary phase, however, the integral will have a much larger value due to the contribution in the proximity of this stationary phase. This is the intuition behind the choice of a resonant $\Omega$. More precisely, the phase becomes stationary when
$$\frac{d}{d\tau}\big(\Omega \tau-l\psi(\tau)\big)=0,\;\;i.e.,\; \Omega=l\omega(E(\tau^*)),$$ 
where $\tau^*$ is a stationary point of phase.
This gives the value of resonant $\Omega$. We now further understand its details.

To do so, 
we first define a notion of intrinsic frequency. Consider the uphill hetericlinic orbit equation: 
\begin{equation}\label{HQ}
\ddot{q}_{h}-\gamma\dot{q}_{h}+V'(q_{h})=0, \;\;\; q_{h}(-\infty)=x_{a},\;\; q_{h}(+\infty)=x_{u}.
\end{equation}
By linearizing $V'(q)$ around $q=x_{st}$, we obtain
\begin{equation*}
\ddot{q}_{h}-\gamma\dot{q}_{h}+V''(x_{a})q_{h}=0,
\end{equation*}
whose characteristic equation has eigenvalues $r=\frac{\gamma}{2}\pm i\sqrt{V''(x_{a})-\frac{\gamma^2}{4}}$. Note that there are two eigenvalues as long as the imaginary part is nonzero (note $\gamma$ is small), corresponding to the general solutions
$$q_h(t) \approx Me^{\frac{\gamma}{2}t}e^{i\sqrt{V''(x_{a})-\frac{\gamma^2}{4}}}+Ne^{\frac{\gamma}{2}t}e^{-i\sqrt{V''(x_{a})-\frac{\gamma^2}{4}}},$$
in which $M, N$ are arbitrary constants. These solutions in general describe oscillation at frequency $\omega=\sqrt{V''(x_{a})-\frac{\gamma^2}{4}}.$ We call $\omega_0=\sqrt{V''(x_{a})}$ the intrinsic frequency as we're using $\gamma\to 0$ asymptotics, although denoting the intrinsic frequency by $\omega_0=\sqrt{V''(x_{a})-\frac{\gamma^2}{4}}$ will not affect the results either. 

In \cref{HQ}, the heteroclinic orbit $q_h(t)$ circles around the metastable state $x_{a}$ (corresponding to minimal $E_m=V(x_{a})$) 
with intrinsic frequency $\omega_0$ for infinite amount of time (see \cref{figEHo} for an illustration). In this phase, the slowly changing energy $E(t)$ is $E(\tau^*)=E_m$. 
We thus obtain resonant frequency $\Omega=l\omega(E(\tau^*))=l\omega_0$. 


We then estimate $|I_1(\Omega)|$. Let $\bar{\psi}(\tau)=\Omega \tau-\psi(\tau)$ which has a stationary point at $\tau=\tau^*$ with $$\bar{\psi}'(\tau^*)=0, \bar{\psi}''(\tau^*)=-\frac{d\omega(E)}{dE}\frac{dE}{d\tau}|_{\tau=\tau^*}\approx-\frac{d\omega(E)}{dE}I(E)
\omega(E)|_{\tau=\tau^*}\neq0.$$ Evaluating the integral $I_1(\Omega)$ by the method of stationary phase (see e.g., \cite{tao2004lecture}; see also \cref{AB}), we get
\begin{equation}\label{AS}
|I_1(\Omega)|\sim\left[\frac{1}{\sqrt{\gamma}}|a_1(E)|\sqrt{\frac{2\pi}{\left|\frac{d\omega(E)}{dE}I(E)\omega(E)\right|}}\right]_{\tau=\tau^*}.
\end{equation}
The symbol `$\sim$' means that the left and right hand sides agree at the leading order in an asymptotic
expansion in $\gamma$. 
This quantitative result shows that, for example, smaller friction coefficient corresponds to a bigger change of transition rate induced by the parametric excitation.

\begin{remark}\label{LinearResonance}
Similar to linear resonance already considered in literature (see, e.g. \cite{Dykman1997Resonant}), the parametric resonant frequency also corresponds to intrinsic frequency $\omega_0$. This may sound inconsistent with the parametric resonant frequency of linear (e.g., $\ddot{q}_h+\omega_0^2(1+\epsilon \cos(\Omega t))q_h=0$) or weakly nonlinear systems (e.g., \cite{Surappa2018Analysis, tao2019simply}) which is $\Omega=2\omega_0$, but the latter is in fact, as discussed above, a special case where $a_1=a_{-1}=0$. As the potential of our system is in general arbitrarily nonlinear, all harmonics could exist (i.e., none of $q_n$'s vanishes). For example, if $q_h(t)=\cos(t)+\cos(2t)$, then this happens and the fundamental frequency of $\dot{q}_h q_h$ is in fact $1$, not $2$, just like that of $\dot{q}_h$ (which corresponds to linear resonance). However, parametric resonance is often more prominent than linear resonance, measured in terms of peak sharpness (defined in \cref{Example2}) for some special models, and we will see this numerically.
\end{remark}

\begin{remark}\label{FreH}
For multi-dimensional case (e.g. $f(q,t)_j=A_j\cos(\Omega t)q_j$), by applying  Theorem \ref{Thm2} and \cref{R3}, the change of transition rate is expressed by 
\begin{equation*}
 \delta S_e=-2\left|\sum_{j=1}^{nd}A_j\int_{-\infty}^{\infty}\dot{q}_{h}^j(t)q_h^j(t)e^{i\Omega t}dt\right|,   
\end{equation*}
where $q_h^j(t)$ (resp. $\dot{q}_h^j(t)$) denotes the $j$th component of $q_h(t)$ (resp. $\dot{q}_h(t)$). As in \cref{AIE}, we further express $q_h^j(t)$ and $\dot{q}_h^j(t)$ as modulated Fourier series, for $j=1,...,nd$. After substitution into $\delta S_e$, we can conduct estimation on integral $\delta S_e$ to understand resonant frequency of parametric resonance. Different from single particle case, i.e. $n=d=1$, Hess$V(x_a)$ will have multiple eigenvalues, and each will give a possible resonant frequency (the strength of each depends on the detailed interactions of coefficients, which is problem dependent, and hence no general claim will be stated). This will be verified numerically.
\end{remark}


\section{Experimental results}
\label{sec:experiments}
We now perform numerical experiments on specific models to 
illustrate our theoretical results.
\subsection{Example 1: Double-well potential}\label{Example1}


As a first test, consider a single particle $q$ moving in
a one dimension potential
$$V(q)=\frac{(1-q^2)^2}{4}.$$
For this example, $n=d=1, \Gamma=\gamma$. The potential $V(q)$ has two wells of equal depth, situated at $q_{a}=-1$, $q_{b}=1$. Saddle point exists at $q_u=0$. We use this example to explore the effect of parametric  forcing $f(q,t)=A\cos(\Omega t)q$ on metastable transition rate from $q_a$ to $q_b$ in light of \cref{R3}.

\paragraph{Parametric resonance} We first approximate the heteroclinic orbit of the forceless system by numerically solving the uphill equation \cref{uphill2}. More precisely, we take the time-reversed uphill equation with a sign flip on velocity,  and simulate the ODE. Since this is a second order boundary value problem and the boundary points at $t=\pm\infty$ incur numerical difficulty, we make an approximation by choosing an initial $q,p$ infinitesimally away from the saddle point, in the direction of the stable eigenvector of the uphill vector field linearized at the saddle $q_u=0$, then simulate an initial value problem using 4th-order Runge-Kutta (RK4) for long enough with sufficiently small time step, and finally collect the result backward in time.
\cref{figEHo} shows an obtained heteroclinic connection from $q_a=-1$ to $q_u=0$ with friction coefficient $\gamma=0.1$ in phase space. Since $q_a$ is a fixed point, the path circles around it for arbitrarily long time.
\begin{figure}\label{figEHo}
\centering
 \includegraphics[width=0.6\textwidth]{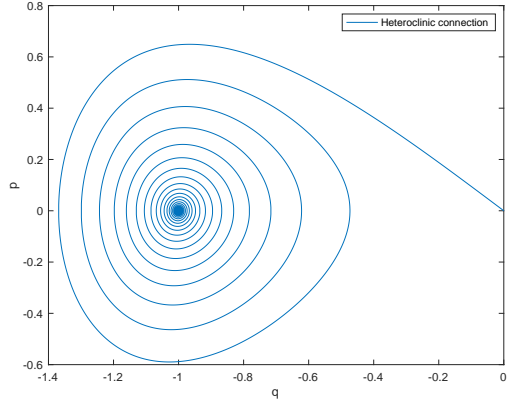}
 \caption{ (Color online)
A heteroclinic connection (or maximum likelihood path ) from $-1$ to $0$ in phase space: $\gamma=0.1$, $p=\dot{q}$ denotes momentum of particle.}
 \end{figure}

With the unforced heteroclinic orbit, now we can examine the dependence of $\delta S_e$ on input frequency $\Omega$. For convenience, denote  $|\int_{-\infty}^{\infty}\dot{q}_{h}(t)q_h(t) e^{i\Omega t}dt|$ by $|I(\Omega)|$. For each $\Omega$, we compute $|I(\Omega)|$ by numerically approximating the integral via piecewise trapezoidal quadrature with high enough resolution.
\cref{Figaction} shows the relationship between the leading order correction to action $|I(\Omega)|$ and $\Omega$ for parameter $\gamma=0.1$.
We observe that, there exists special $\omega^*$ at which $|I(\Omega)|$ peaks, corresponding to the resonant frequency in our theoretical discussion. More details now follow:


\paragraph{Parametric resonant frequency} By the theoretical analysis conducted in 
\cref{sec: PR}, the exact parametric resonant frequency is intrinsic frequency  $\omega_0=\sqrt{V''(q_{a})}\approx1.4142$, with which heteroclinic orbit $q_h(t)$ oscillates around metastable state $q_{a}$. Consistent with it, as shown in \cref{Figaction}, the function $|I(\Omega)|$ displays a sharp peak near $\omega_0$, and sequentially weaker peaks near its integer multiples. This is numerical evidence that the resonant frequencies are related to the intrinsic frequenct $\omega_0$ of the unperturbed system \cref{uphill2}. 
\begin{figure}
\centering
 \includegraphics[width=0.6\textwidth]{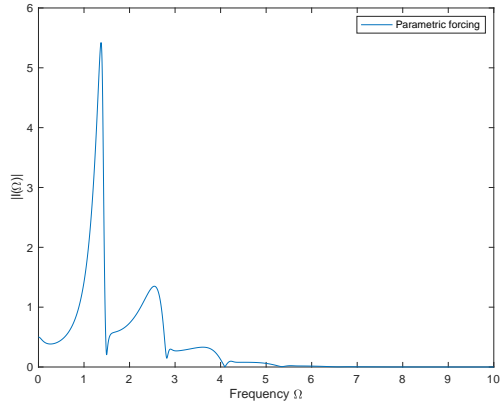}
 \caption{ (Color online)
$\delta S=-2A\epsilon |I(\Omega)|$: damping $\gamma=0.1$. The dependence of  action correction $|I(\Omega)|$ of the double-well system on frequency $\Omega$ in the case of parametric forcing.}
  \label{Figaction}
 \end{figure}
 
\paragraph{Estimation of $I(\Omega)$ near a resonant frequency} We proceed to depict the dependence of $|I(\Omega)|$ on $\Omega$ with $\gamma$ fixed as $\gamma=0.1,\;0.01,\; 0.001$, which is shown in \cref{Fig3gamma}. As we can see, smaller values of $\gamma$ lead to more prominent peaks, i.e. near resonant frequency $\omega^*=\omega_0$, the value of $|I(\omega^*)|$ is larger when $\gamma$ decreases. One can further find that $|I(\omega^*)|$ increases by a factor of $3.2160$ when $\gamma$ varies from $0.1$ to $0.01$ or a factor of $3.33$ from $0.01$ to $0.001$, by comparing values of $|I(\omega^*)|$ (marked in \cref{Fig3gamma} with arrows) 
corresponding to $\gamma=0.1, 0.01, 0.001$. Interestingly, such numerical relation between $|I(\omega^*)|$ and $\gamma$ satisfies $I(\omega^*)\sim\sqrt{\frac{1}{\gamma}}K$ approximately. This scaling with $\gamma$ well agrees with our stationary phase estimate \cref{AS}.
\begin{figure}
\centering
 \includegraphics[width=0.6\textwidth]{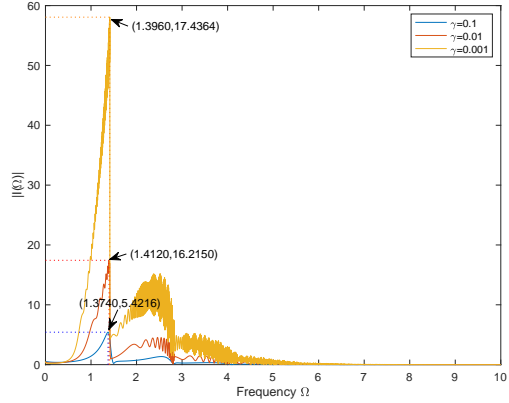}
 \caption{ (Color online)
$\delta S=-2A\epsilon |I(\Omega)|$. The dependence of  action correction $|I(\Omega)|$ of the stochastic double-well system on frequency $\Omega$ for different damping $\gamma=0.1, 0.01, 0.001$, respectively.}
  \label{Fig3gamma}
 \end{figure}


\subsection{Example 2: Nonlinear pendulum (periodic potential)}\label{Example2}
To further test our theoretical results, we now consider an even more nonlinear potential, 
$$V(q)=\sin q.$$
Here $n=d=1$. We will also use this example to illustrate the differences between linear resonance and parametric resonance.

Focusing on a compact neighborhood in which this potential has two local minima located in $q_a=-\frac{\pi}{2}$, $q_b=\frac{3}{2}\pi$, a saddle point located in $q_u=\frac{\pi}{2}$ separates their basins of attraction. 
Consider the two special forms of forcing $f(q,t)$ discussed in \cref{R3}, the first being a linear forcing $f(t)=A\cos(\Omega t)$, and the second being a parametric forcing $f(q,t)=A\cos(\Omega t)q$.

In order to compare quantitatively, let us introduce the notion of peak sharpness, which is defined as the change ratio of $S_e(\Omega)$ in $\Omega$, namely,
\begin{equation}\label{rate}
\rho_p(\Omega)=\frac{\left|\int_{-\infty}^{\infty}\dot{q}_{h}(t)q_h(t)e^{i\Omega t}dt\right|}{\left|\int_{-\infty}^{\infty}\dot{q}_{h}(t)q_h(t)e^{i(\Omega+d\Omega) t} dt\right|},\; \rho_l(\Omega)=\frac{\left|\int_{-\infty}^{\infty}\dot{q}_h(t)e^{i\Omega t}dt\right|}{\left|\int_{-\infty}^{\infty}\dot{q}_h(t)e^{i(\Omega+d\Omega) t} dt\right|},
\end{equation}
where $d\Omega$ is an infinitesimal increment. If $\rho_p(\omega^*)>\rho_l(\omega^*)$ holds, it means that parametric excitation at a resonant frequency lead to a sharper peak than that of linear excitation, and we utilize it as a basis to check if parametric resonance is more apparent than linear resonance.

\begin{figure}
 \centering
 \subfigure[]{
 \includegraphics[width=0.4\textwidth]{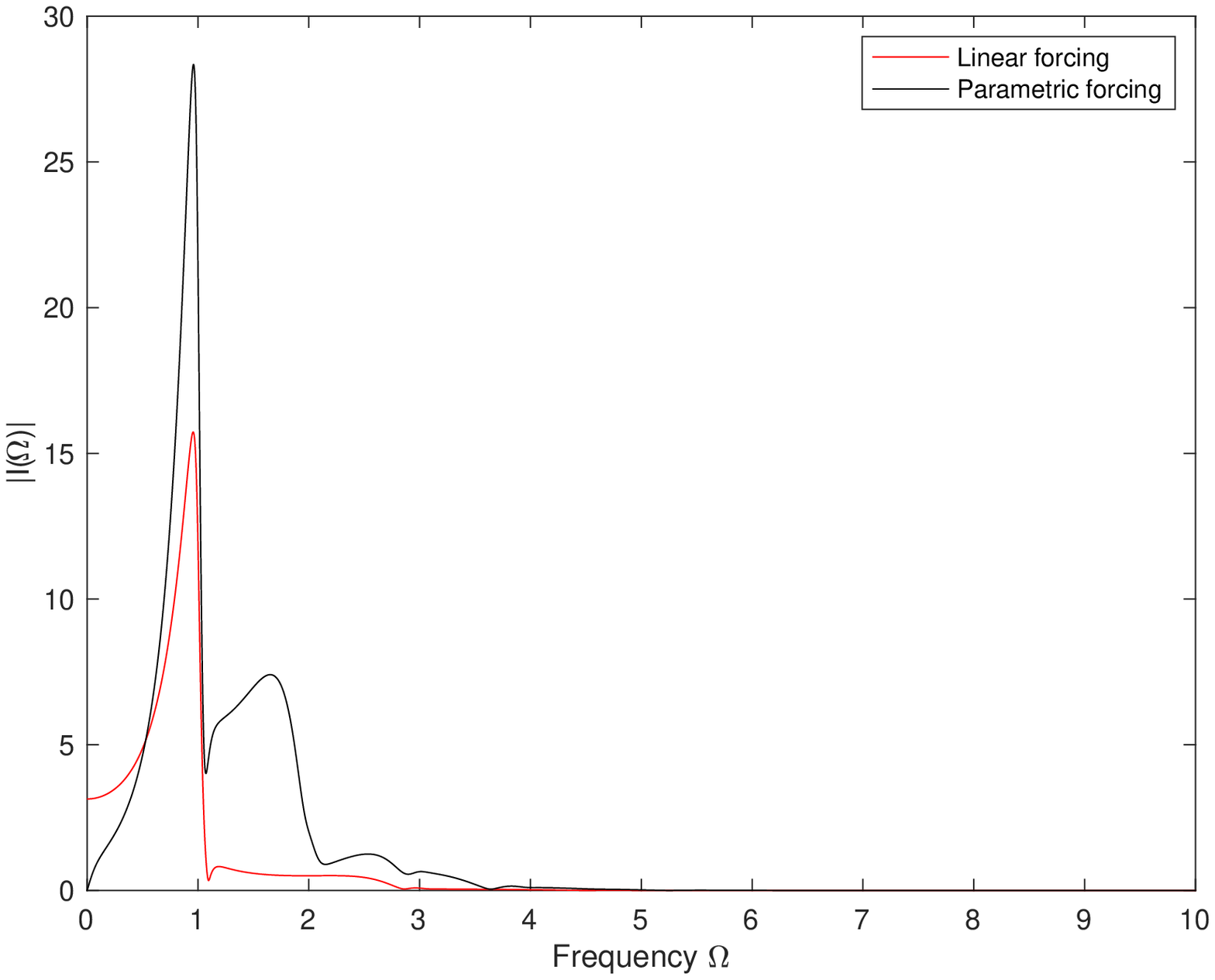}
 }
 \subfigure[]{
 \includegraphics[width=0.4\textwidth]{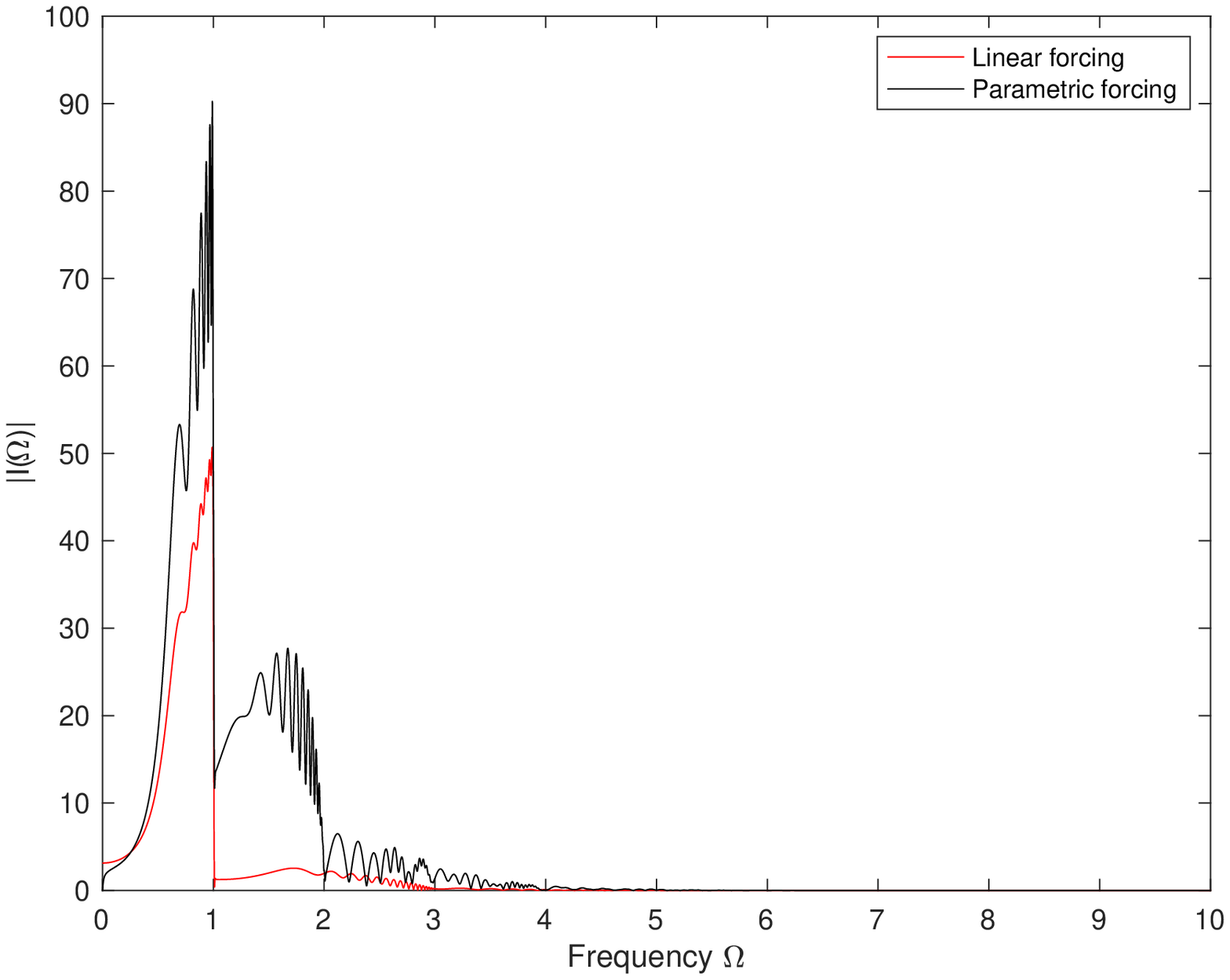}
 }
  \caption{ (Color online) $\delta S=-2A\epsilon |I(\Omega)|$. The dependence of  action correction $|I(\Omega)|$ of system on frequency $\Omega$ in two special cases, respectively. (a) damping $\gamma=0.1$, (b) damping $\gamma=0.01$.  }
 \label{figaction-sinq}
 \end{figure}

As in \cref{Example1}, numerically computed $|I(\Omega)|$ for $\gamma=0.1, 0.01$ is plotted in \cref{figaction-sinq}, respectively. Again, the main peaks of $|I(\Omega)|$ correspond to intrinsic frequency $\omega^*\approx\omega_0\approx1$, both in the case of additive and parametric forcing. In terms of \cref{rate}, we can compute the $\rho_l(\omega^*)$, $\rho_p(\omega^*)$ both for $d\Omega=0.01,\;-0.01$ and $\gamma=0.1,\;0.001$, and list them in \cref{Tab:01}. According to these data, it is interesting to see that the peak of parametric resonance is sharper. One can further find that $\rho_p(\omega^*)$ varies more greater than that of $\rho_l(\omega^*)$, as $\gamma$ decreases from $0.1$ to $0.01$. The results in this example seems to suggest parametric resonance is often more prominent than linear resonance in terms of peak sharpness. 


\begin{table}\caption{Values of peak sharpness}
\centering
\label{Tab:01}
\begin{tabular}{cccc}
  \hline
  Infinitesimal  increment  &  Cases & $\rho_p(\omega^*)$ & $\rho_l(\omega^*)$ \\
  \hline
  \multirow{3}{*}{d$\Omega$=0.01}  &$\gamma=0.1$ & 1.0118 & 1.0094\\
  \cline{2-4}
   &$\gamma=0.01$ & 2.0212& 1.2706\\
  \cline{1-4}
 \multirow{3}{*}{d$\Omega=-0.01$} & $\gamma$=0.1 & 1.0077  &1.0060\\
  \cline{2-4}
& $\gamma$=0.01& 1.3957& 1.0657\\
  \cline{1-4}
  \hline
\end{tabular}
\end{table}

\subsection{Example 3: Lennard-Jones molecular cluster}\label{Example3}

Finally, let us consider a practical application, for which we apply our techniques 
to a multi-particle molecular system. 
Based on  Theorem \ref{Thm2}, \cref{R3} and \cref{FreH},
we now characterize the parametric resonant frequency in higher dimension case numerically.

We consider $n=36$ molecules in a $d=2$-dimensional periodic box (with box sizes $s_x=3\sqrt{3}$, $s_y=6$ in x-, y- directions respectively). The $j$th molecule's location is denoted by $q^{(j)}=(x^{(j)},y^{(j)})^T\in (\mathbb{R}/s_x)\times (\mathbb{R}/s_y)$. The governing dynamics is  
\begin{equation}\label{Ex2}
\begin{split}
  &\ddot{x}^{(j)}+\gamma\dot{x}^{(j)}=-\frac{\partial}{\partial x^{(j)}}V_{LJ}(\cdot)+\varepsilon A_1\cos(\omega t)x^{(j)}+\sqrt{\mu} \gamma^{\frac{1}{2}}\xi_x^{(j)}(t),\\
  &\ddot{y}^{(j)}+\gamma\dot{y}^{(j)}=-\frac{\partial}{\partial y^{(j)}}V_{LJ}(\cdot)+\varepsilon A_2\cos(\omega t)y^{(j)}+\sqrt{\mu} \gamma^{\frac{1}{2}}\xi_y^{(j)}(t),\end{split}
\end{equation}
for $j=1,\cdots, 36$. We use the notation $$r_{ij}=\left(\left|\text{mod}\left(x^{(i)}-x^{(j)}+\frac{s_x}{2},s_x\right)-\frac{s_x}{2}\right|^2+\left|\text{mod}\left(y^{(i)}-y^{(j)}+\frac{s_y}{2},s_y\right)-\frac{s_y}{2}\right|^2\right)^{\frac{1}{2}}$$
to denote the distance between the $i$th and $j$th molecules under periodic boundary condition (i.e., geodesic distance on the 2-torus). 
$V_{LJ}$ is Lennard-Jones potential which is widely used in molecular modeling, and it is the sum of pairwise interactions,
$$V_{LJ}(r)=\sum_{i\neq j, i,j=1}^n\left[\left(\frac{r_0}{r_{ij}}\right)^{12}-2\left(\frac{r_0}{r_{ij}}\right)^{6}\right],$$
where $r_0$ is a constant parameter denoting the characteristic distance of particles, taken as $r_0=1$ here. 
This potential has a lot of local minima, and for an important material sciences application, we consider a global minimum $q_b$ corresponding to a perfect lattice configuration, and a local minimum $q_{a}$ corresponding to material with a local defect, and we are interested in how to turn the material from the defective state $q_{a}$ to the perfect state $q_b$. In addition, there is a saddle point at $q_s$ on their separatrix between $q_a$ and $q_b$, and these fixed points are depicted in \cref{fig3}. At 
the minima $V(q_a)\approx-109.7064$, $V(q_b)\approx-120.4712$, and at
saddle $V(q_s)\approx-106.8218$. In this case, increasing the metastable transition rate from $q_a$ to $q_b$ is of particular importance, as it corresponds to healing the defect of the material. This transition is still a rare event, but we will see its likelihood can be significantly increased by an appropriate homogeneous external vibration (i.e., shaking the material to perfect its lattice).

\begin{figure}
 \centering
 \subfigure[]{
 \includegraphics[width=0.4\textwidth]{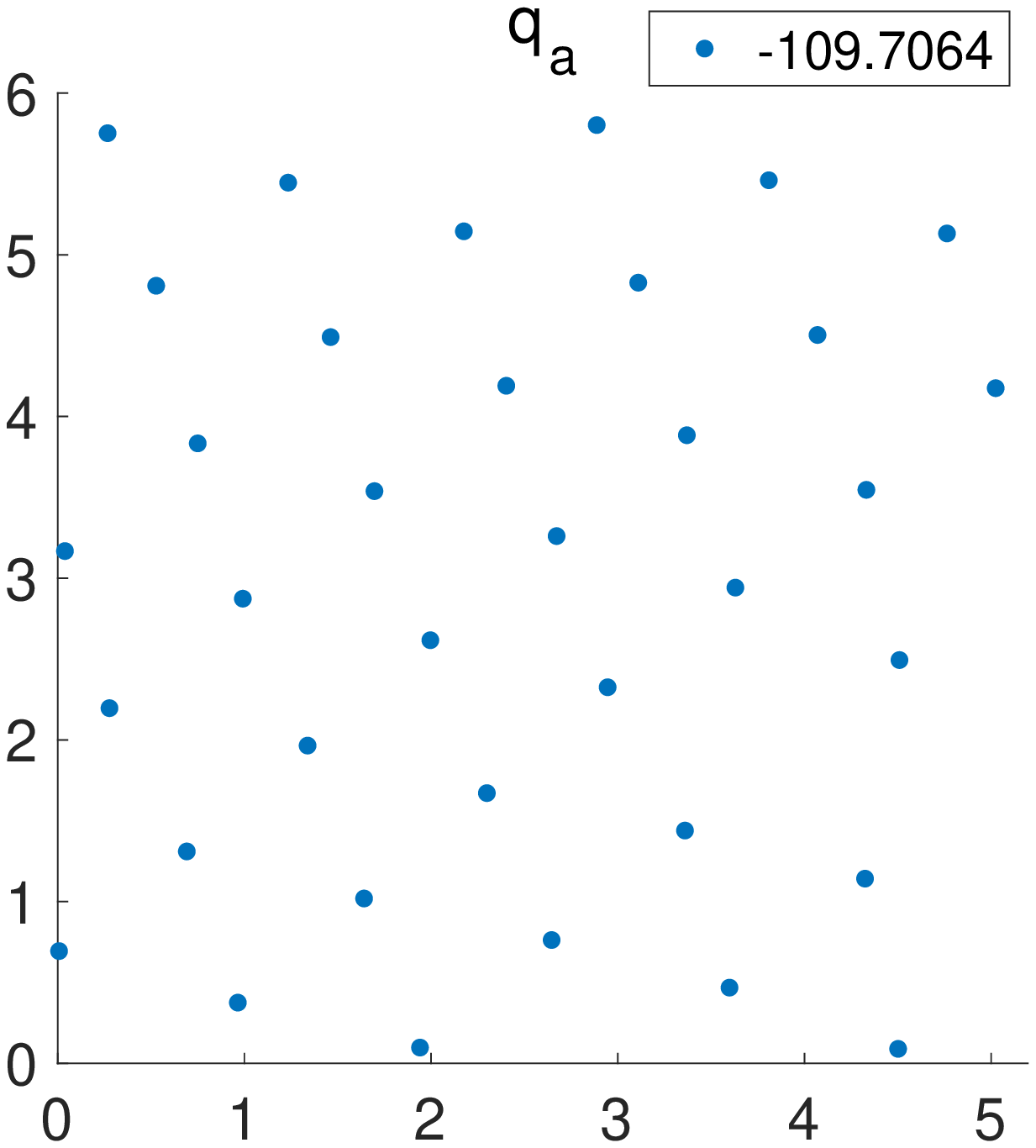}
 }
 \subfigure[]{
 \includegraphics[width=0.4\textwidth]{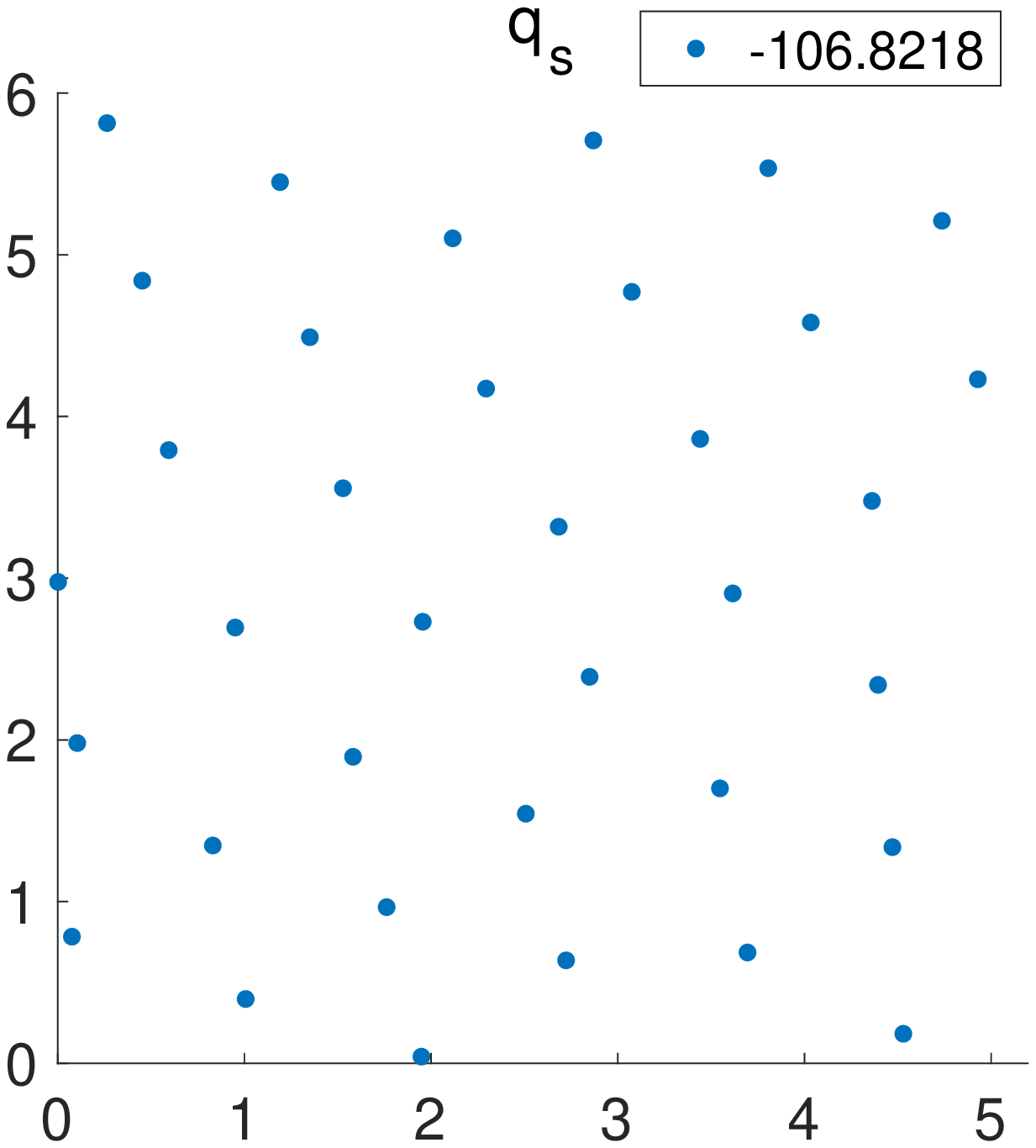}
 }
 \subfigure[]{
 \includegraphics[width=0.4\textwidth]{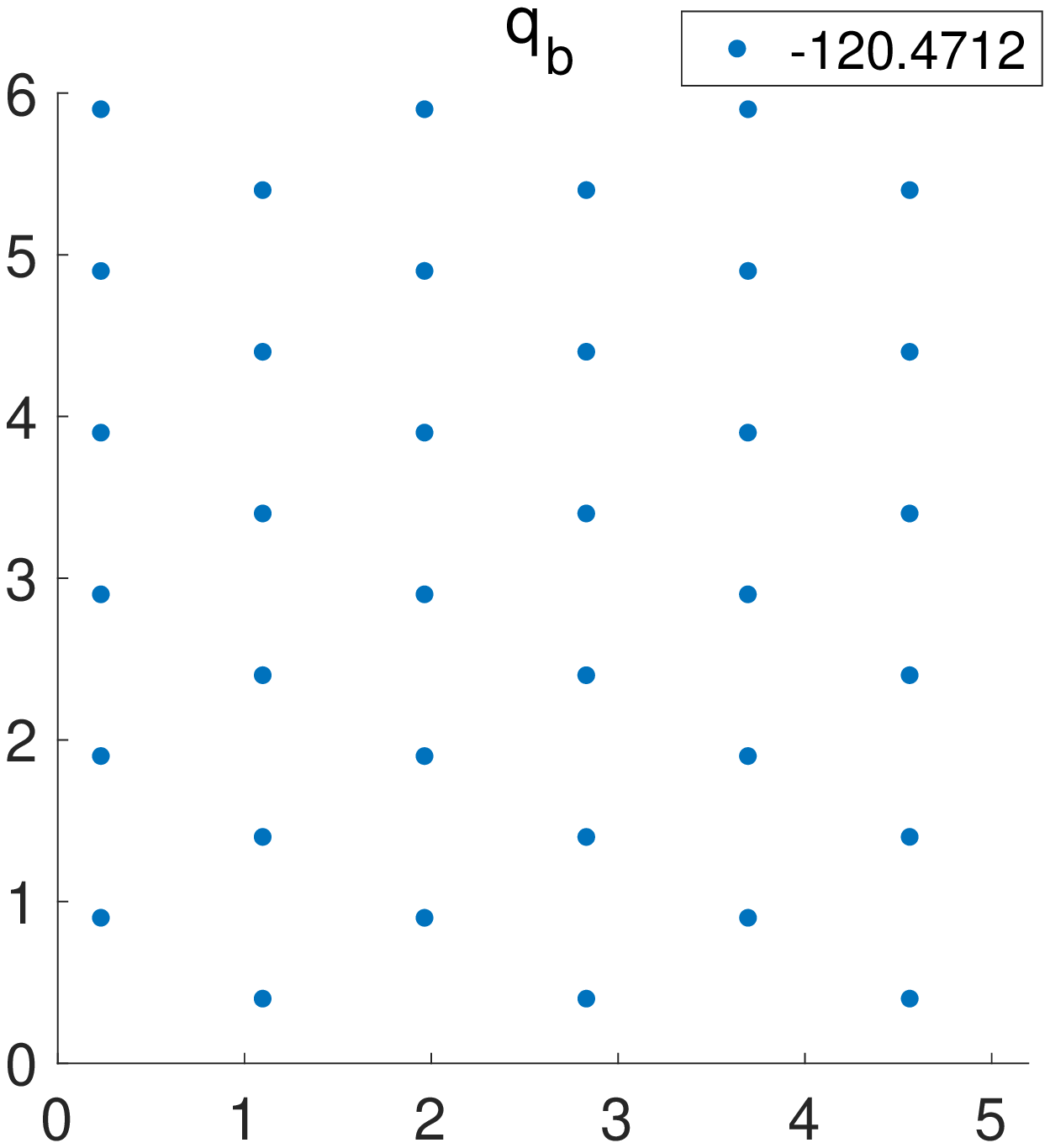}
 }
  \caption{ (Color online) Configurations and corresponding potential values: $r_0=1$. (a) Initial configuration, (b) Saddle configuration, (c) Final configuration.}
 \label{fig3}
 \end{figure}


For the case of parametric perturbation discussed here, by applying \cref{Thm2} and \cref{R3}, the change of the transition rate from $q_a$ to $q_b$ is given by a more simple form: 
 $$\delta S_e=\left\{
\begin{aligned}
&-2A_1\left|\sum_{j=1}^{n}\int_{-\infty}^{\infty}\dot{x}^{(j)}(t)x^{(j)}(t)e^{i\omega t}dt\right|\quad\;\;\;\;\;\;\;\;\;\;\;\;\;\;\;\;\;\;\;\;\;\;\;\;\;\;if~A_1\neq0, \;A_2=0;\\
&-2A_2\left|\sum_{j=1}^{n}\int_{-\infty}^{\infty}\dot{y}^{(j)}(t)y^{(j)}(t)e^{i\omega t}dt\right|  \quad \;\;\;\;\;\;\;\;\;\;\;\;\;\;\;\;\;\;\;\;\;\;\;\;\;\;if~A_1=0,\;A_2\neq 0; \\
&-2A\left|\sum_{j=1}^{n}\int_{-\infty}^{\infty}\left(\dot{x}^{(j)}(t)x^{(j)}(t)+\dot{y}^{(j)}(t)y^{(j)}(t)\right)e^{i\omega t}dt\right|~if~A_1=A_2=A\neq0.
\end{aligned}
\right.$$
\noindent As in \cref{Example1}, let us first compute the heteroclinic connection \cref{uphill2} from $q_a$ to $q_s$ numerically. Then the application we need to do is to find the optimal frequency $\omega^*$, vibrating $q_a$ into $q_b$ through $q_s$, to achieve the purpose of heal defect. 


\paragraph{Parametric resonant frequency} For simplicity, let $|F(\omega)|$ denote the $|\cdot|$ part in above formula. We numerically computed the heteroclinic orbit in the unforced system, which gives $x,y,\dot{x},\dot{y}$, and then evaluate $|F(\omega)|$ via quadrature over a range of $\omega$ values. The results of the three different parametric forcing cases, respectively corresponding to vibrating in the x-, y-, and both directions, are plotted in \cref{fig4} for damping coefficient $\gamma=1$. 
\begin{figure}
 \centering
 \subfigure[]{
 \includegraphics[width=0.4\textwidth]{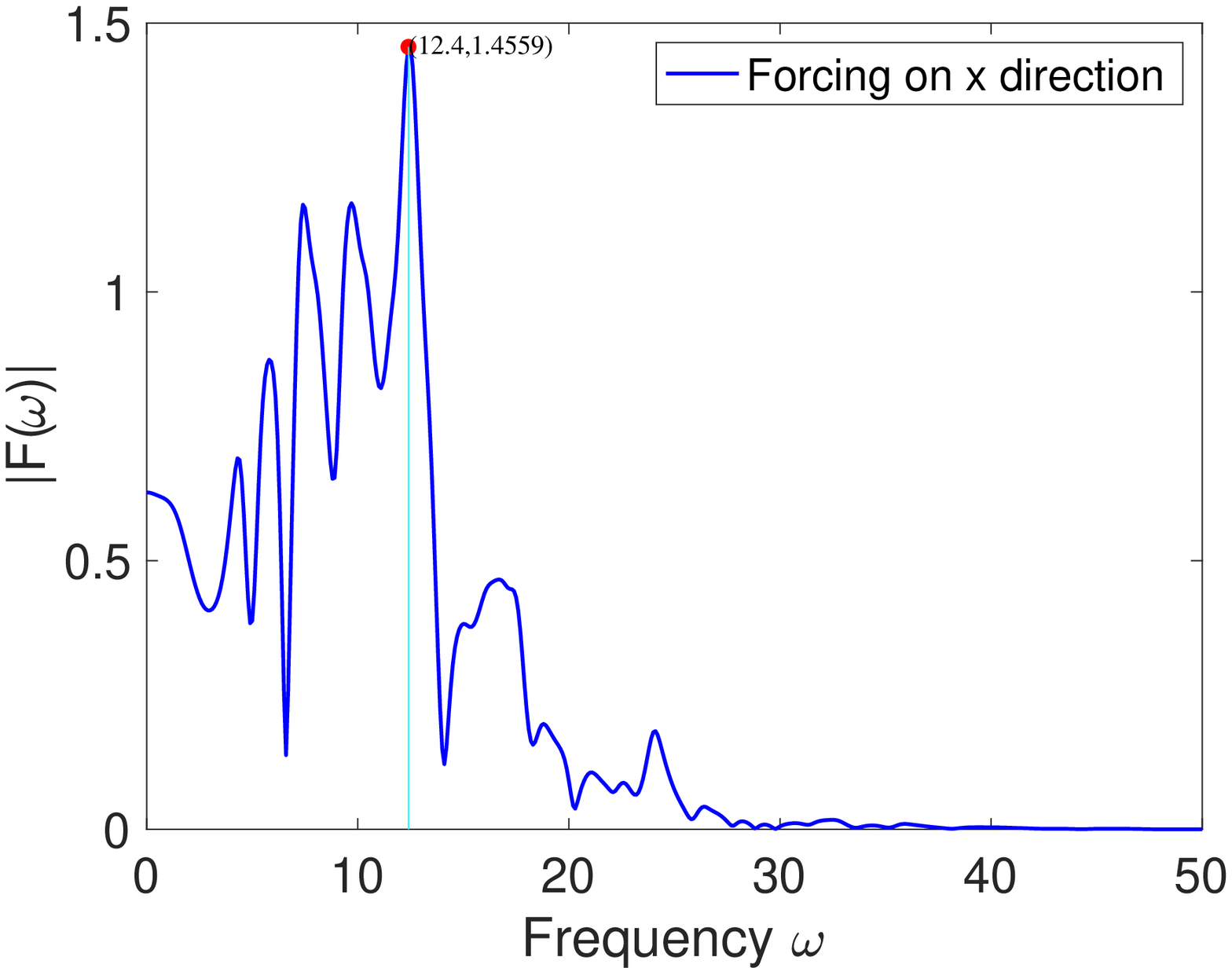}
 }
 \subfigure[]{
 \includegraphics[width=0.4\textwidth]{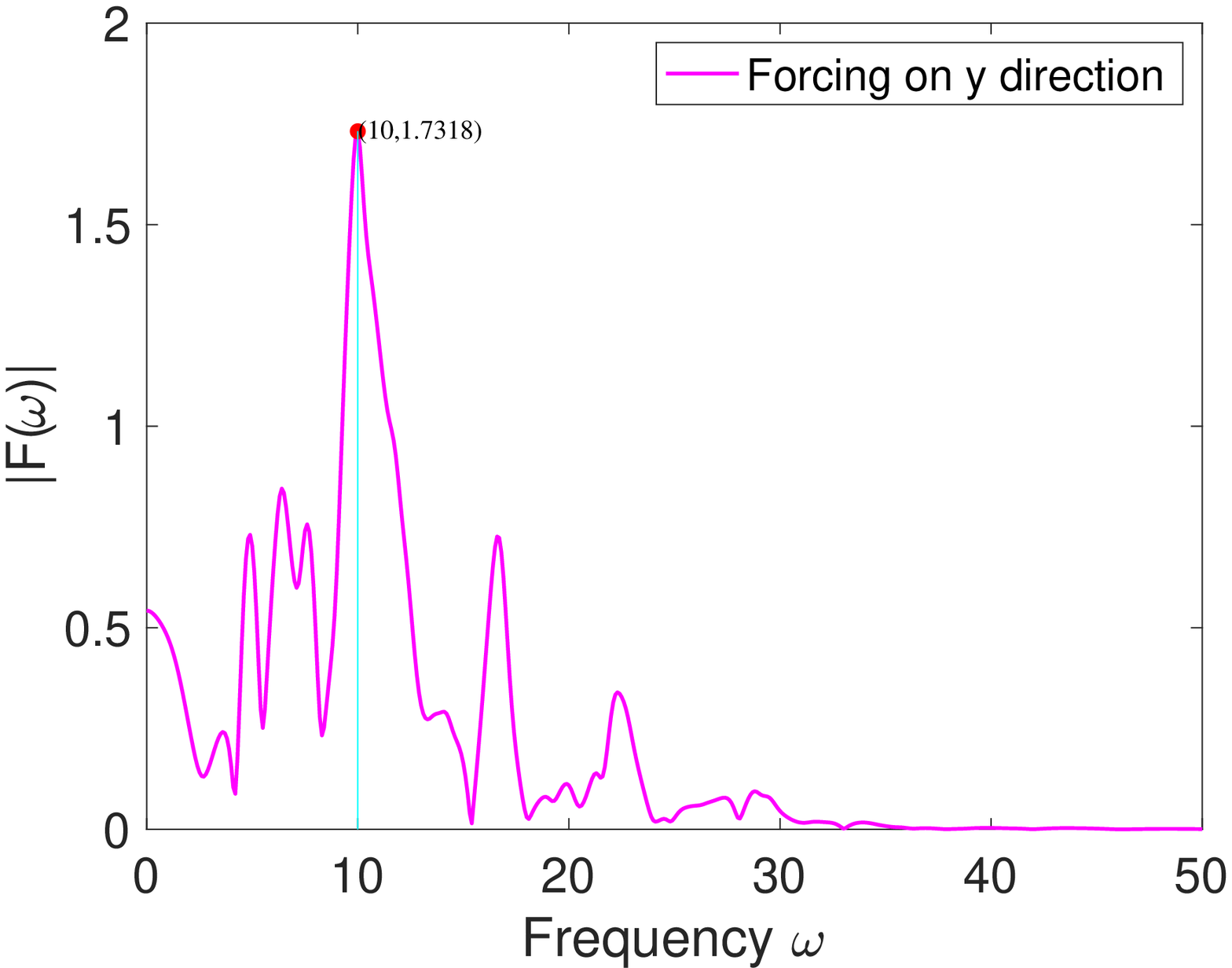}
 }
 \subfigure[]{
 \includegraphics[width=0.4\textwidth]{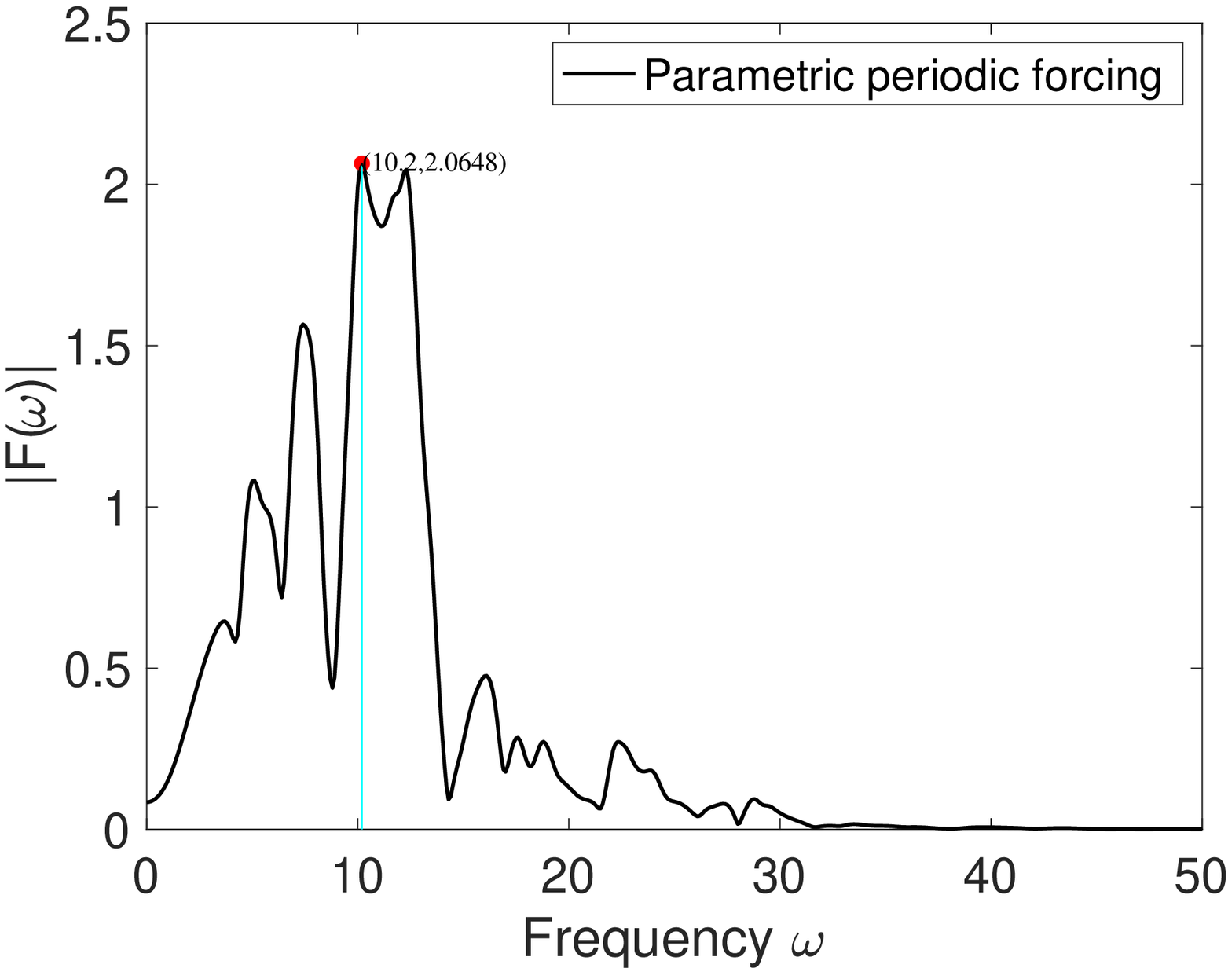}
 }
  \caption{ (Color online) $\delta S_e=-2A|F(\omega)|$: damping $\gamma=1$. The dependence of  action correction $|F(\omega)|$ of \cref{Ex2} on frequency $\omega$ in three special cases, respectively.}
 \label{fig4}
 \end{figure}
We again see that $|F(\omega)|$ displays clear peaks. Different from the problems of single particle in one dimension, Hess$V(q_a)$ is now a $72\times72$ matrix with multiple eigenvalues instead of just one. By examining the list of eigenvalues, we see that resonant frequencies again coincide with eigenvalues of the matrix Hess$V(q_a)$. The strongest resonant frequency is marked in each plot. Therefore, to heal a defective material, one possibility is to use our theory and compute the resonant frequencies, and then try vibrations at those frequencies. Of course, given this is a high dimensional system, there are many different ways to combine vibrations at each dimension; if one wants to optimize the combination, our theory can also help and one no longer has to conduct computationally expensive rare event simulations, but this becomes an optimization problem which deserves an adequate investigation in a different study.

\paragraph{Comparison to linear forcing}
The rest of this subsection is devoted to a comparison to the case of linear perturbation; a clear advantage of parametric forcing will be illustrated. Specifically, the governing dynamics for the case of linear perturbation is  
\begin{equation}\label{Ex2linear}
\begin{split}
  &\ddot{x}^{(j)}+\gamma\dot{x}^{(j)}=-\frac{\partial}{\partial x^{(j)}}V_{LJ}(\cdot)+\varepsilon A_1\cos(\omega t)+\sqrt{\mu} \gamma^{\frac{1}{2}}\xi_x^{(j)}(t),\\
  &\ddot{y}^{(j)}+\gamma\dot{y}^{(j)}=-\frac{\partial}{\partial y^{(j)}}V_{LJ}(\cdot)+\varepsilon A_2\cos(\omega t)+\sqrt{\mu} \gamma^{\frac{1}{2}}\xi_y^{(j)}(t),\end{split}
\end{equation}
for $j=1,\cdots, 36$. Again, based on \cref{Thm2} and \cref{R3}, the change of the transition rate from $q_a$ to $q_b$ is written in a more simple form : 
 $$\delta S_e=\left\{
\begin{aligned}
&-2A_1\left|\sum_{j=1}^{n}\int_{-\infty}^{\infty}\dot{x}^{(j)}(t)e^{i\omega t}dt\right|\quad\;\;\;\;\;\;\;\;\;\;\;\;\;\;\;\;\;\;\;\;\;\;\;\;\;\;if~A_1\neq0, \;A_2=0;\\
&-2A_2\left|\sum_{j=1}^{n}\int_{-\infty}^{\infty}\dot{y}^{(j)}(t)e^{i\omega t}dt\right|  \quad \;\;\;\;\;\;\;\;\;\;\;\;\;\;\;\;\;\;\;\;\;\;\;\;\;\;if~A_1=0,\;A_2\neq 0; \\
&-2A\left|\sum_{j=1}^{n}\int_{-\infty}^{\infty}\left(\dot{x}^{(j)}(t)+\dot{y}^{(j)}(t)\right)e^{i\omega t}dt\right|\;\;\;\;\;\;\;\;\;\;\; if~A_1=A_2=A\neq0.
\end{aligned}
\right.$$
Let $|F(\omega)|$ still denote the $|\cdot|$ part in above formula. The frequency response results of the three different linear forcing cases, respectively corresponding to vibrations in the x-, y-, and both directions, are plotted in \cref{7a} for damping coefficient $\gamma=1$. One may again try to identify special $w^*$ values at which $|F(\omega)|$ peaks, but these peaks are not as well-defined as that in the parametric resonance case. In fact, note the drastic difference between $|F|$ values in the parametric case ($\sim 1$) and this (linear) case ($\sim 10^{-14}$). We feel there is no strong resonance in this case any more, and integrals cancel out so that the computed $\delta S_e$ is dominated by (small) numerical errors. To illustrate this cancellation, the plots of $Re[\int_{-\infty}^{\infty}\dot{x}^{(j)}(t)e^{i\omega t}dt]$ (denoted by $H(\omega)$) as functions of $\omega$ for several different j's are also provided in \cref{7b}.



This is empirical evidence of the advantage of parametric excitation, at least it leads to resonant enhancement of the recovery of material defect. 
\begin{figure}
\centering
\subfigure[]{
\begin{minipage}[b]{0.4\textwidth}
\includegraphics[width=\textwidth]{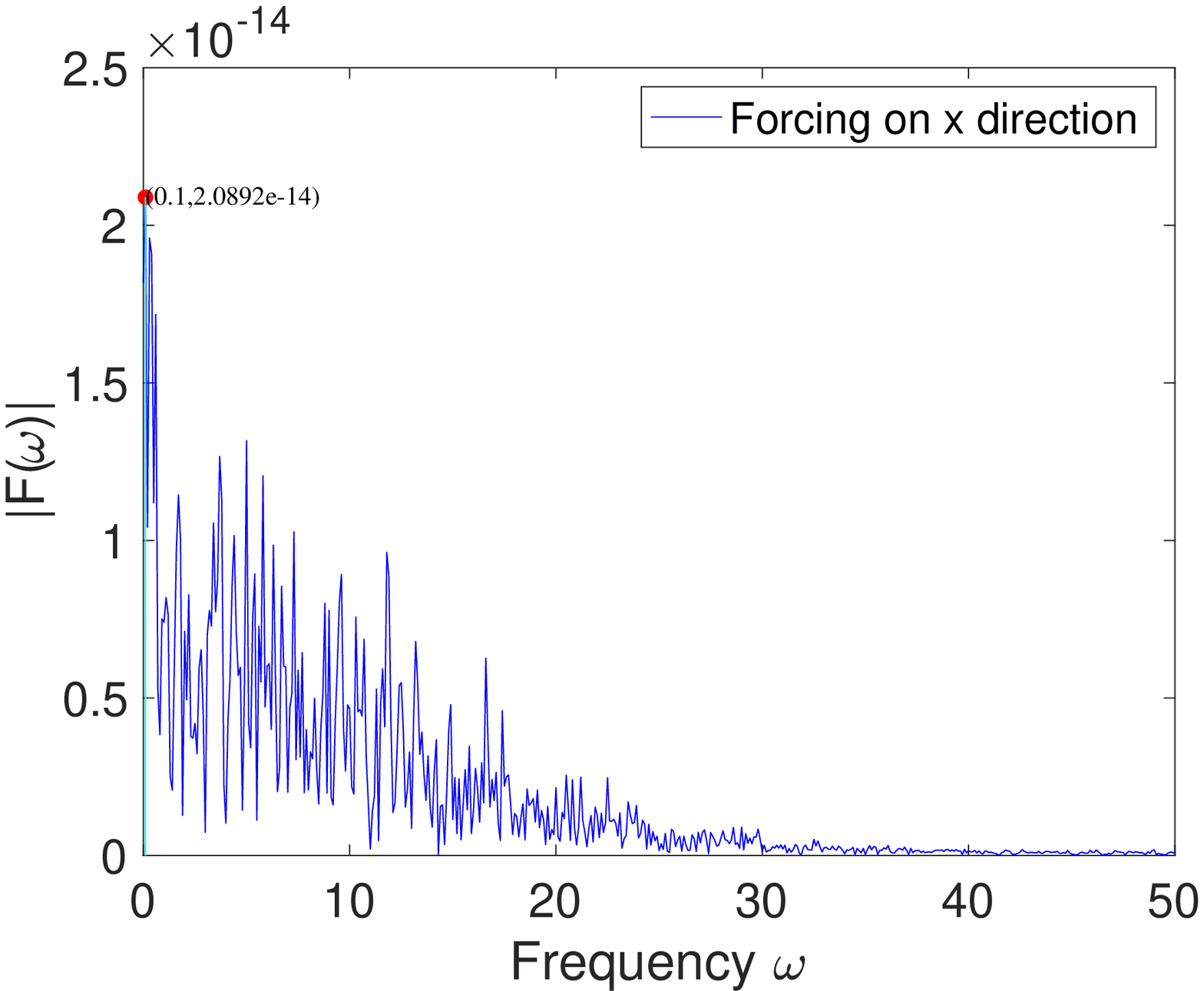} \\
\includegraphics[width=\textwidth]{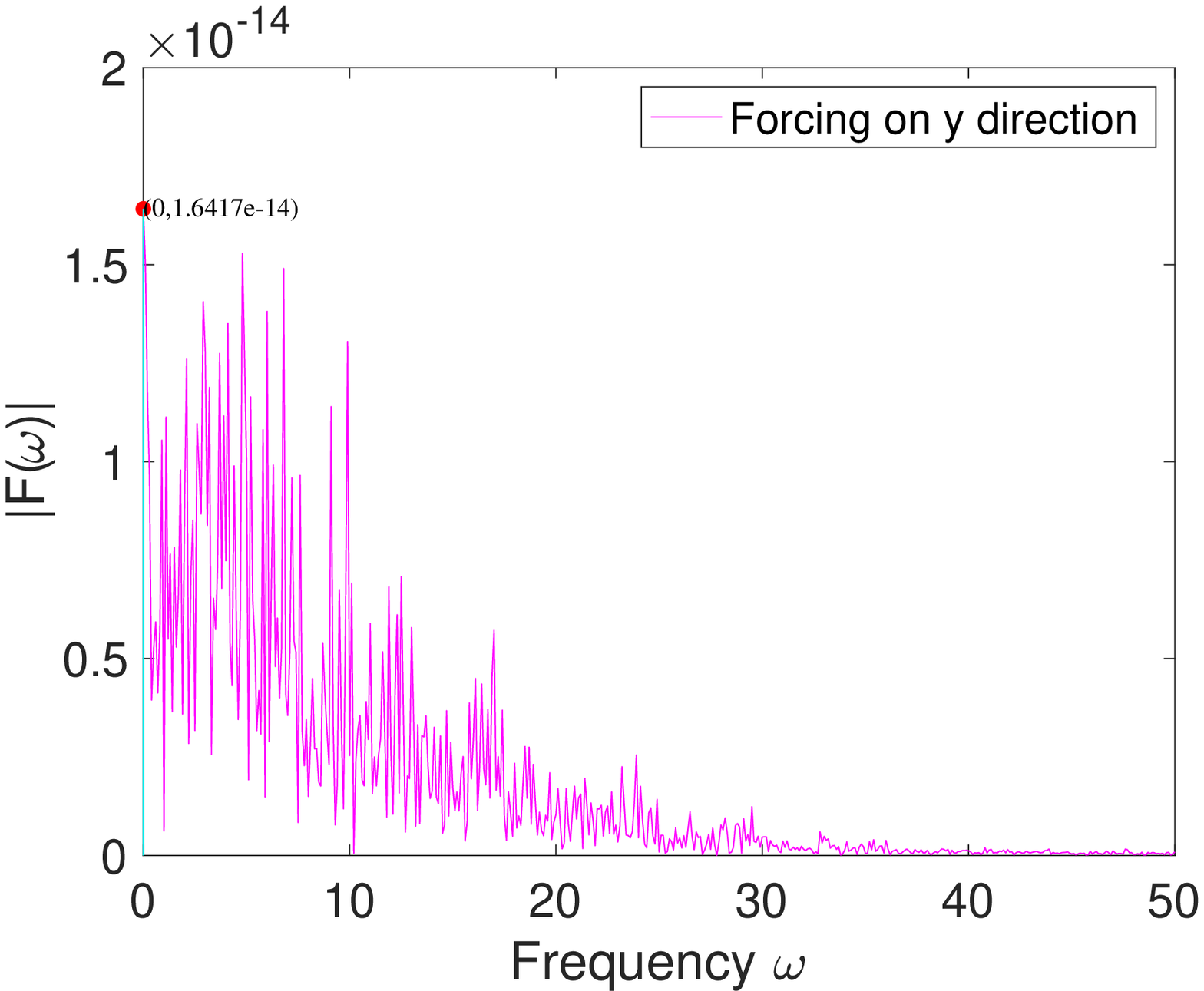}\\
\includegraphics[width=\textwidth]{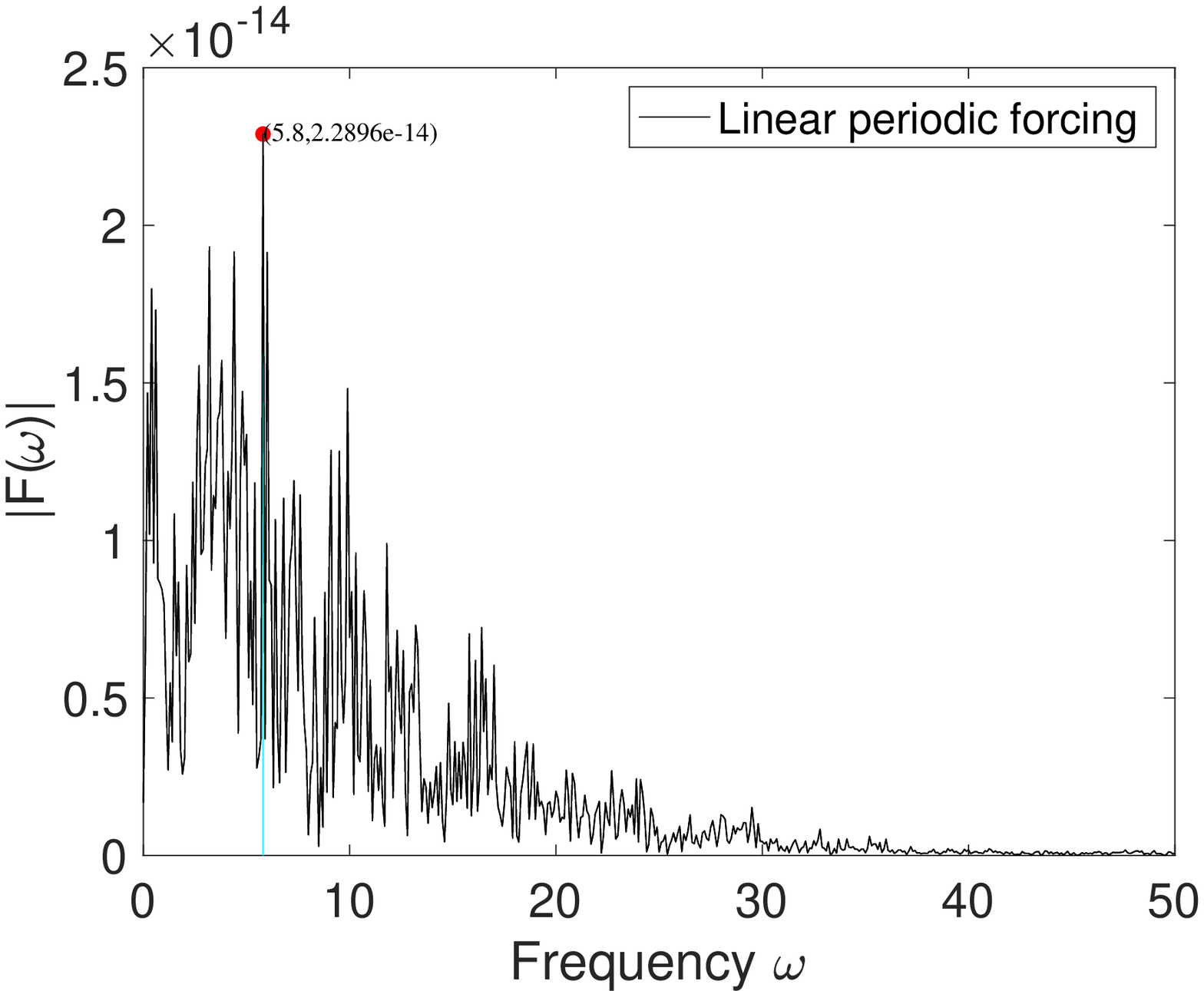}
\end{minipage}
}\label{7a}
\subfigure[]{
\begin{minipage}[b]{0.4\textwidth}
\includegraphics[width=\textwidth]{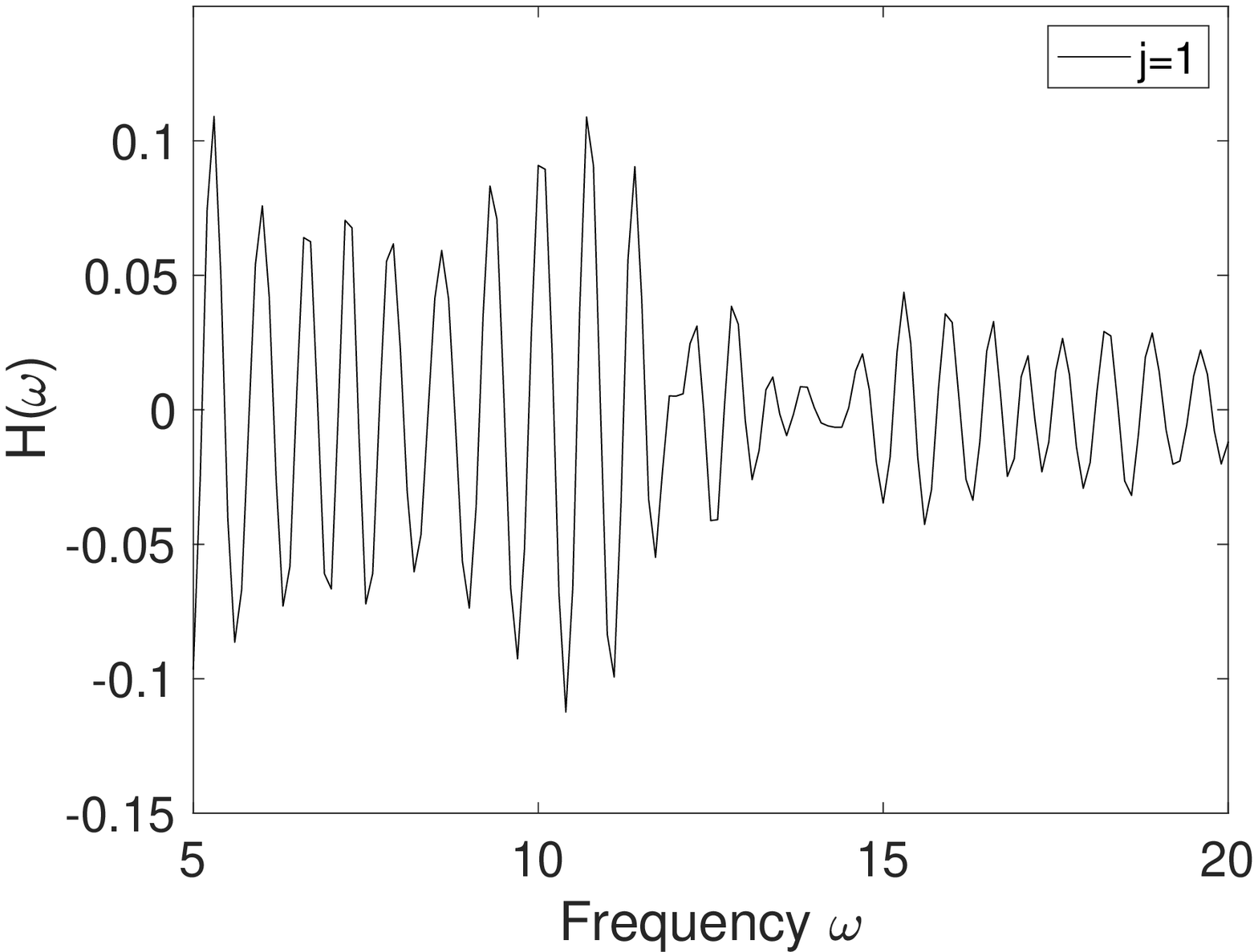} \\
\includegraphics[width=\textwidth]{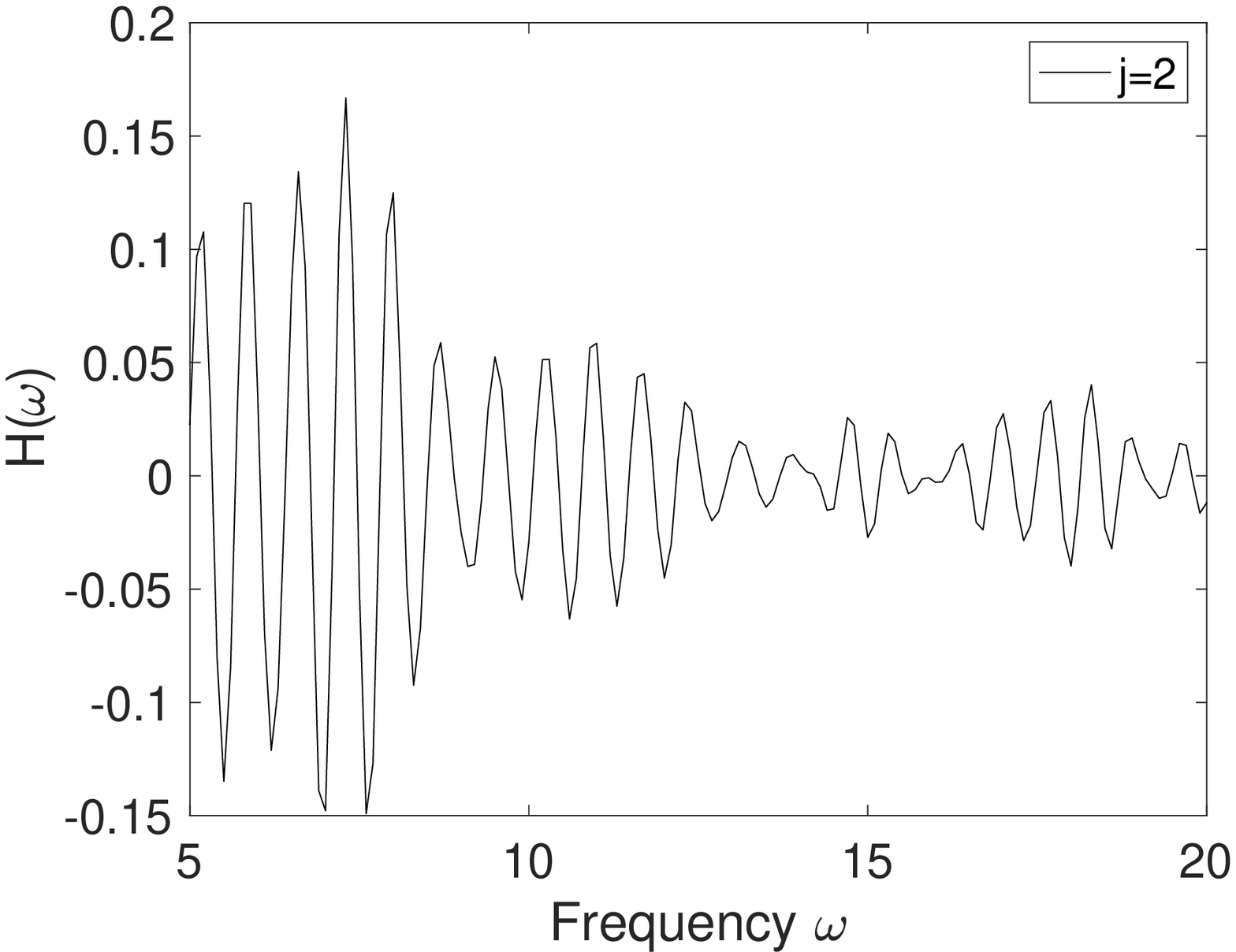}\\
\includegraphics[width=\textwidth]{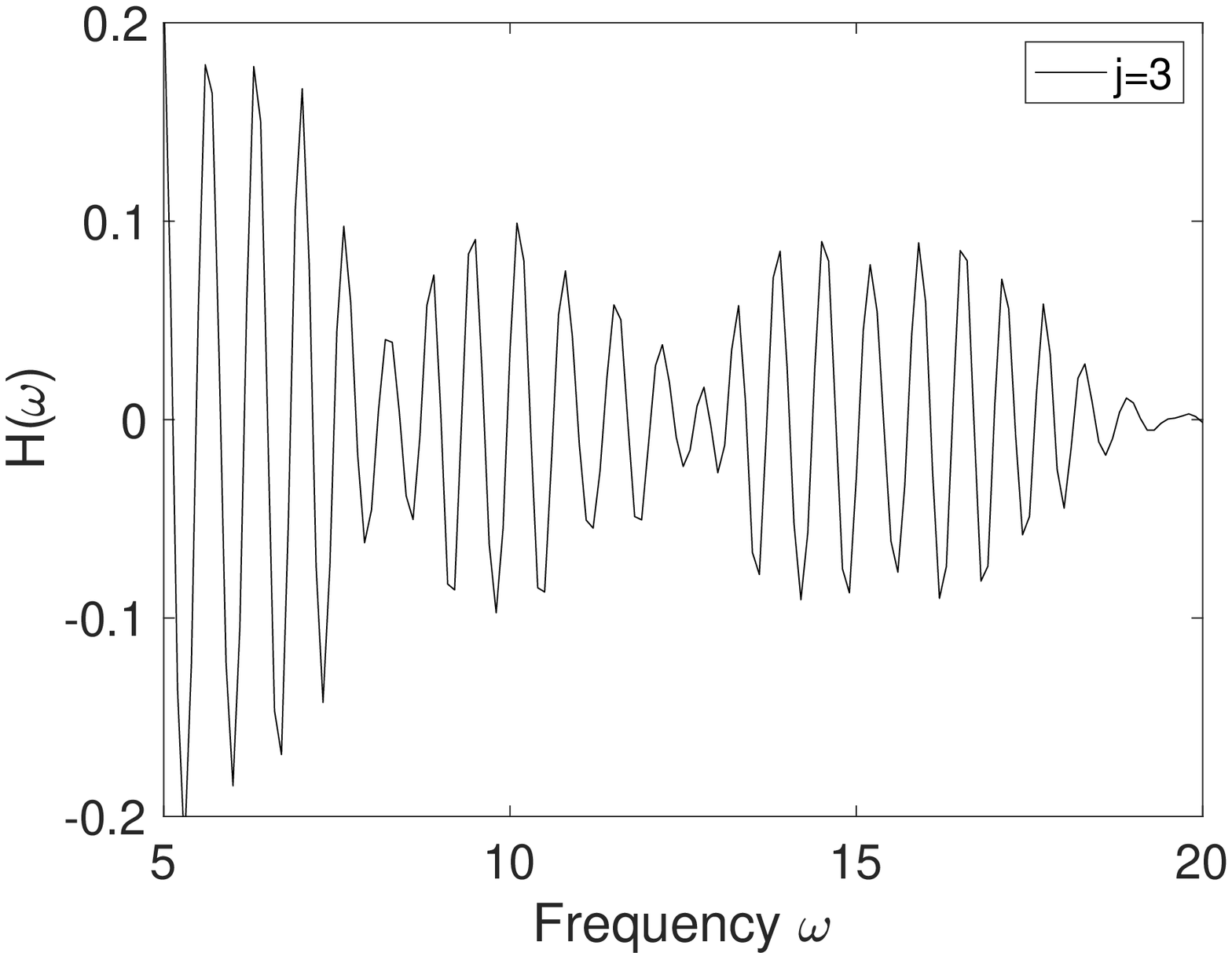}
\end{minipage}
}\label{7b}
\caption{(Color online) Damping $\gamma=1$. (a) $\delta S_e=-2A\epsilon |F(\omega)|$. The dependence of  action correction $|F(\omega)|$ on frequency $\omega$ in three special linear forcing cases, respectively. (b) The dependence of components $Re[\int_{-\infty}^{\infty}\dot{x}^{(j)}(t)e^{i\omega t}dt]$ denoted by $H(\omega)$ on frequency $\omega$, for $j=1,2,3$. Note we zoomed-in the x-axis for improved readability. This is reasonable since we just need to show the cancellation here and the plots need not to be very complete.
}
\end{figure}


\section{Conclusion}
\label{sec:conclusions}
In this work, we derived a closed-form explicit expression that characterizes how a small, generic nonlinear periodic forcing affects the metastable transition rate in kinetic Langevin systems of arbitrary dimensions. This is done by viewing the high-order Euler-Lagrange equations associated with the Freidlin-Wentzell action minimization in the perspective of perturbed Hamiltonian dynamics. Perturbation analysis allows the MLP and its rate to be approximated from the heteroclinic connection in the unperturbed, noiseless system. Furthermore, we showed that parametric periodic perturbation facilitates metastable transitions by theoretically characterizing the resonant frequency of parametric excitation via stationary phase asymptotics. Numerical experiments for both  low-dimensional toy models and a 144-dimensional molecular cluster validated our theory. The method we developed here could offer insights to the interaction between periodic force and noise in rather general systems.


\appendix
\section{Euler-Lagrangian Equations}\label{sec:EL}
Consider the variational problem of minimizing the action functional 
$$S[x]=\int_{t_0}^{t_f}L(t, x, \dot{x},\ddot{x})dt$$
over the set of paths $x\in C^1([t_0, t_f], \mathbb{R}^n)$ satisfying the boundary conditions 
$$x(t_0)=x_{t_0},\;\;\; x(t_f)=x_{t_f}.$$
A path $x\in C^1([t_0, t_f], \mathbb{R}^n)$ from $x_{t_0}$ to $x_{t_f}$ is said to be minimal if $S[x]\leq S[x+\xi]$ for every variation $\xi\in C^1([t_0, t_f], \mathbb{R}^n)$ such that $\xi(t_0)=\xi(t_f)=0$, $\dot{\xi}(t_0)=\dot{\xi}(t_f)=0$.
\begin{lemma}\label{ELE}
A minimal path $x$: $[t_0,t_f]\to \mathbb{R}^n$ is a solution to the Euler-Lagrange equations 
\begin{equation}\label{eq:aa}
\frac{\partial\mathcal{L}}{\partial x}-\frac{d}{dt}\frac{\partial\mathcal{L}}{\partial\dot{x}}+\frac{d^2}{dt^2}\frac{\partial\mathcal{L}}{\partial\ddot{x}}=0.
\end{equation}
\end{lemma}
\begin{proof}
Assume that $x$ is minimal. Thus all directional derivatives of $S$ at $x$ vanish, i.e., 
\begin{align}\label{action-AR}
0=&\frac{d}{d\eta}|_{\eta=0}S(x+\eta \xi)=\frac{d}{d\eta}|_{\eta=0}\int_{t_0}^{t_f}L(t, x+\eta \xi, \dot{x}+\eta \dot{\xi},\ddot{x}+\eta \ddot{\xi})dt\notag\\
&=\int_{t_0}^{t_f}(\sum_{i=1}^{n}\frac{\partial L}{\partial x_i}(t, x, \dot{x},\ddot{x})\xi_i+\sum_{i=1}^{n}\frac{\partial L}{\partial \dot{x}_i}(t, x, \dot{x},\ddot{x})\dot{\xi}_i+\sum_{i=1}^{n}\frac{\partial L}{\partial \ddot{x}_i}(t, x, \dot{x},\ddot{x})\ddot{\xi}_i) dt\notag\\
&=\int_{t_0}^{t_f}(\sum_{i=1}^{n} \frac{\partial L}{\partial x_i}-\frac{d}{dt}\frac{\partial L}{\partial \dot{x}_i}+\frac{d^2}{dt^2}\frac{\partial L}{\partial \ddot{x}_i})\xi_i dt
\end{align}
for all variations  $\xi\in C^1([t_0, t_f], \mathbb{R}^n)$ with $\xi(t_0)=\xi(t_f)=0$, $\dot{\xi}(t_0)=\dot{\xi}(t_f)=0$. Here the last equality is based on integration by parts and the boundary conditions for $\xi$.
\end{proof}

\section{Brief review of the method of stationary phase}\label{AB}The method of stationary phase \cite{tao2004lecture} established integral asymptotics \cref{AS} and proposed to determine the leading-order behavior of the integral
\begin{equation}\label{integralI}
I(\nu)=\int_{a}^{b}f(t)e^{i\nu g(t)}dt    
\end{equation}
for $\nu\gg1$, where functions $f$ and $g$ are smooth enough to admit Taylor approximations near some appropriate point in $[a,b]$, and $g$ is real-valued.\par
We assume that $g'(c)=0$ at some point $c\in(a,b)$, and that $g'(t)\neq0$ everywhere else in the closed interval. Assume further that $g''(c)\neq0$ and $f(c)\neq0$. Let $\sigma$ be the sign of $g''(c)$. Then 
$$\sigma g''(c)=|g''(c)|.$$
Thus, to leading order,
\begin{equation}\label{SPA}
 I(\nu)\sim f(c)e^{i\nu g(c)}\sqrt{\frac{2\pi}{\nu|g''(c)|}}e^{\frac{\pi i \sigma}{4}}, \;as\;\;\nu\to\infty.   
\end{equation}
The symbol `$\sim$' is used to mean that the right-hand side is the first term in an asymptotic
expansion of the left-hand side.
Equation \cref{SPA} is called the stationary phase approximation, due to the fact that the main contribution to the integral comes form a region of a point $c$ at which the phase $g(t)$ is stationary. For more details on derivation of \cref{SPA}, see \cite{tao2004lecture}.
\section*{Acknowledgments}
The authors would like to thank Professor Jinqiao Duan, Dr Pingyuan Wei and Dr Yang Li for helpful discussions. This work was primarily done while YC was a visiting scholar at Georgia Institute of Technology. The research of YC was supported by the NSFC grant 12101484. MT is grateful for partial support by NSF DMS-1847802.

\bibliographystyle{siamplain}
\bibliography{yingreferences}
\end{document}